 \documentclass[final,3p,times]{elsarticle}
\usepackage{amsmath}
\usepackage{amssymb}
\usepackage{amsthm}

\newcommand{\Om}{\Omega}
\newcommand{\Ga}{\Gamma}
\renewcommand{\vec}[1]{\mathbf{#1}}

\newtheorem{theorem}{Theorem}[section]

\newtheorem{lemma}[theorem]{Lemma}
\newtheorem{proposition}[theorem]{Proposition}
\DeclareMathOperator*{\esssup}{ess\,sup}
\DeclareMathOperator*{\essinf}{ess\,inf}
\DeclareMathOperator*{\aufspann}{span}

\begin{document}
\begin{frontmatter}

%% Title, authors and addresses

%% use the tnoteref command within \title for footnotes;
%% use the tnotetext command for the associated footnote;
%% use the fnref command within \author or \address for footnotes;
%% use the fntext command for the associated footnote;
%% use the corref command within \author for corresponding author footnotes;
%% use the cortext command for the associated footnote;
%% use the ead command for the email address,
%% and the form \ead[url] for the home page:
%%
 %%\title{Mathematical analysis of a general marine ecosystem model with nonlinear coupling terms and non-local boundary conditions\tnoteref{label1}}
%% \tnotetext[label1]{}
%% \author{Name\corref{cor1}\fnref{label2}}
%% \ead{email address}
%% \ead[url]{home page}
%% \fntext[label2]{}
%% \cortext[cor1]{}
%% \address{Address\fnref{label3}}
%% \fntext[label3]{}

\title{Mathematical analysis of a marine ecosystem model with nonlinear coupling terms and non-local boundary conditions}

%% use optional labels to link authors explicitly to addresses:
%% \author[label1,label2]{<author name>}
%% \address[label1]{<address>}
%% \address[label2]{<address>}

%\author{}
\author[cau]
{Christina~Roschat\corref{cor1}}
\ead{cro@informatik.uni-kiel.de}
\author[cau]
{Thomas~Slawig}
%\ead{ts@informatik.uni-kiel.de}
\address[cau]
{Department of Computer Science, Kiel Marine Science - Centre for Interdisciplinary Marine Science, Christian-Albrechts-Universit\"at zu Kiel,
24098 Kiel, Germany
}
\cortext[cor1]{Corresponding author.
}
%\fntext[dfg]{Supported by the DFG Cluster Future Ocean.}

\begin{abstract}
We investigate the weak solvability of initial boundary value problems associated with an ecosystem model of the marine phosphorus cycle. The analysis covers the model equations themselves as well as their linearization which is important in the model calibration via parameter identification. We treat both cases simultaneously by investigating a system of advection-diffusion-reaction equations coupled by general reaction terms and boundary conditions.
We derive a weak formulation of the generalized equations and prove two theorems about its unique solvability provided that the reaction terms consist of Lipschitz continuous and monotone operators. In the proofs, we adapt different techniques (Galerkin approximation, Banach's Fixed Point Theorem) to the multi-dimensional model equation. By applying the general theorems to the problems associated with the phosphorus model we obtain results about existence and uniqueness of their solutions. Actually, by assuming a generalized setting the theorems establish the basis for the mathematical analysis of the whole model class to which the investigated phosphorus model belongs.
\end{abstract}

%\begin{keyword}
%Biogeochemical modeling \sep Initial boundary value problems \sep Advection-diffusion-reaction equations \sep Weak solutions \sep Nonlinear coupling \sep Non-local boundary conditions
%
%% MSC codes here, in the form: 
%\MSC[2010] 35Q92 \sep 35K61 \sep 35D30
%% or \MSC[2008] code \sep code (2000 is the default)

%\end{keyword}

\end{frontmatter}
\thispagestyle{empty}

%%
%% Start line numbering here if you want
%%
% \linenumbers

%% main text

\section{Introduction}
The understanding of biogeochemical cycles in marine ecosystems is an important aspect in many scientific areas. 
In climate research, the oceans are investigated because of the prominent role they play in the global cycle of carbon dioxide ($CO_2$).
The greenhouse gas is taken up through the sea surface into the water and passes through a cycle which is basically determined by the transformation of $CO_2$ into organic material during the photosynthesis of marine plants and its remineralization after their dying. To some extend, the dead organic material remains on the sea bottom effecting a long-term storage of $CO_2$. This effect is supposed to help understand and control climate change (see e.g. \cite{dil03}).

Mathematical ecosystem models give a means to describe biogeochemical cycles in marine ecosystems. They provide information about the concentrations of the involved tracers (molecules or life forms) and thereby contribute to the understanding of the underlying biogeochemical processes. 

An ecosystem model consists of a system of advection-diffusion-reaction equations whose dimension corresponds to the number of tracers. These partial differential equations, also called transport equations, describe, on the one hand, the influence of the ocean circulation on the tracer concentration (see e.g. Stocker \cite{sto11}). In applications, the corresponding values for the current velocity (advection) and the diffusion coefficient are computed by ocean circulation models. Therefore, in the mathematical investigations, they are assumed to be known. 

On the other hand, the model equations are coupled by reaction terms reflecting the biogeochemical processes of the ecosystem.
For example, reaction terms can express predator-prey relationships between two tracers or the growth of phytoplankton depending on insolation and photosynthesis. 

During the investigation of marine ecosystem models the reaction terms are of particular interest. While the terms concerning the ocean circulation are certain reaction terms provide a means to adapt the model to the relevant biogeochemical processes. According to the great variety of possible ecosystems and tracer cycles, many kinds of reaction terms can appear.
They can depend on place and time (e.g. insolation varies over the day) as well as on all tracers to describe their mutual influence. Reaction terms in marine ecosystem models are mostly nonlinear and usually contain a non-local part, i.e. a part that depends on more than one spatial coordinate. Typically, sinking processes over water columns require this kind of reaction term. 
Coupling and nonlinearity pose a challenge to the numerical as well as the theoretical treatment of the model equations.

One of the most important tasks in modeling is to ensure that the model output reflects really observed data as exactly as possible. As soon as the basic structure of the reaction terms is determined this task mainly involves identifying adequate parameter values. Parameters like e.g. growth rates, half saturation constants or remineralization rates are essential for the description of the respective biogeochemical process. Parameter identification is often carried out using optimal control theory, e.g. by R\"uckelt et al. \cite{rue10} who investigate an ecosystem model with four equations. Thereby, the distance between observational data and the model output, regarded as a function of the parameters, is minimized. To characterize the optimal parameters it is useful to formulate an optimality system which contains, beside the original model equation, the so-called adjoint equation coupled by the adjoint operators of the reaction terms' Fr\'echet-derivatives (cf. Tr\"oltzsch~\cite{tr02}). The formulation of the adjoint equation requires the solution of a system coupled by the derivatives of the original reaction terms. Solving this derivative (or linearization) of the original equations is therefore an important step towards adequate parameters.

In applications, biogeochemical models and their derivatives are solved numerically whereas, mostly, the underlying continuous models undergo no further investigation. However, continuous and discretized equations depend on each other. If the continuous equations turned out to be insolvable it would be difficult to interpret the numerically obtained solution. If otherwise the equations were not uniquely solvable different numerical algorithms might yield different, possibly inadequate, solutions. In both cases, the quality of the numerical results would be called into question. 
Thus, the theoretical analysis provides an instrument to validate and improve biogeochemical models.

The main aspects of marine ecosystem models are illustrated by the $PO_4$-$DOP$-model by Parekh at al.~\cite{pa05} describing the marine phosphorus cycle. It is the basis for more complicated models (in the cited paper the authors add an equation to investigate the marine iron cycle) and serves for scientific purposes (testing numerical methods and algorithms). Additionally, the occurring reaction terms are typical for marine ecosystem models.

In this paper, we analyze the existence and uniqueness of weak solutions of both the $PO_4$-$DOP$-model equations and their derivative. 
Since both equations are structured equally we consider a more general setting including generalized reaction terms and boundary conditions. We additionally allow an arbitrary number of model equations such that the results will be applicable also to other, more complicated models.

%%%
%%%
The paper is structured as follows: In the following section, we introduce the mathematical formulation of the $PO_4$-$DOP$-model and specify the initial boundary value problems investigated in this paper. In Sec.~\ref{sec:weak}, we state some mathematical preliminaries and develop a weak formulation of the original problem. The next two sections each contain the formulation and proof of one existence and uniqueness theorem concerning the weak formulation. In Sec.~\ref{sec:examples}, we apply the general results to the problems associated with the $PO_{4}$-$DOP$-model. In the last section, we draw some conclusions from the previous results.

\section{The $PO_4$-$DOP$-model}\label{sec:modellgleichungenndop}
Parekh at al. \cite{pa05} present a model of the iron concentration in relation to the marine phosphorus cycle. Leaving out the iron component leads to a model of the global phosphorus cycle, called $PO_4$-$DOP$-model after the two relevant tracers. The authors, focusing on biogeochemical aspects, confine themselves to briefly outlining the mathematical assumptions about their ecosystem model. Therefore, in the following introduction, we add and precise some information in order to obtain a complete mathematical formulation. 
\subsection{The domain}
The modeled ecosystem is located in a three-dimensional bounded domain $\Om\subseteq\mathbb{R}^3$. $\Om$ is determined by the open, bounded water surface $\Om^{\prime}\subseteq\mathbb{R}^2$ and a well-defined depth $h(x^{\prime})>0$ at every surface point $x^{\prime}\in \Om^{\prime}$. The function $h$ is supposed to be smooth and bounded by the total depth of the ocean $h_{max}$. Thus, we have 
\begin{itemize}
  \item the domain $\Om:=\{(x^{\prime\!\!\!},x_3);x^{\prime}\in\Om^{\prime},x_3\in (0,h(x^{\prime}))\}$ and
  \item the boundary $\Ga:=\{(x^{\prime\!\!\!},h(x^{\prime}));x^{\prime}\in\Om^{\prime}\}\cup(\overline{\Om^{\prime}}\times\{0\})$ consisting of the boundary inside the water and the surface.  
\end{itemize}
The domain is separated into two layers, the euphotic, light-flooded zone $\Om_1$ below the surface and the dark, aphotic zone $\Om_2$ beneath. The maximal depth of the euphotic zone is denoted by $\bar{h}_e\leq h_{max}$. However, it is possible for the domain to end within the euphotic zone.
To cover this case the actual depth of the euphotic zone is defined by $h_e(x^{\prime}):=\min\{\bar{h}_e,h(x^{\prime})\}$, a function of the coordinate $x^{\prime}$. We accordingly split the surface into the part $\Om_2^{\prime}:=\{x^{\prime}\in \Om^{\prime}; h(x^{\prime})>\bar{h}_e\}$ above the aphotic zone and the rest $\Om^{\prime}_1:=\Om^{\prime}\setminus \Om_2^{\prime}$. Analogously dividing the boundary we arrive at
\begin{itemize} 
\item the euphotic zone $\Om_1:=\{(x^{\prime\!\!\!},x_3);x^{\prime}\in\Om^{\prime},x_3\in (0,h_e(x^{\prime}))\}$,
  \item the aphotic zone $\Om_2:=\{(x^{\prime\!\!\!},x_3);x^{\prime}\in\Om_2^{\prime},x_3\in (\bar{h}_e,h(x^{\prime}))\}$,
  \item  the euphotic boundary $\Ga_1:=\{(x^{\prime\!\!\!},h(x^{\prime}));x^{\prime}\in\overline{\Om^{\prime}_1}\}$,
  \item the aphotic boundary $\Ga_2:=\{(x^{\prime\!\!\!},h(x^{\prime}));x^{\prime}\in\Om_2^{\prime}\}$.
\end{itemize}
%
%jetzt die modellgleichungen
%
\subsection{The model equations}
We consider the two tracers phosphate, $PO_4$, and dissolved organic phosphorus, $DOP$, as components of the vector $y:=(y_1,y_2):=(PO_4,DOP)$. Each of the tracers is regarded as a function of space and time solving the non-autonomous advection-diffusion-reaction equation
 \begin{align*}
\displaystyle \partial_ty_j(x,t)+\vec v(x,t)\cdot\nabla y_j(x,t)-{\rm div}(\kappa_j(x,t)\nabla  y_j(x,t))+d_j(y,x,t)=0
\end{align*}
$\text{for all }(x,t)\in\Om\times[0,T]$ and $j=1,2$. 

The velocity $\vec v$ represents advection while $\kappa_j$ is a coefficient for both turbulent and molecular diffusion. 
Since turbulent dominates molecular diffusion the latter is often neglected, i.e. it is assumed $\kappa=\kappa_1=\kappa_2$. However, from a mathematical point of view this simplification is not necessary. 

The biogeochemical processes, represented by the reaction terms $d_j$, differ according to the layers. In the light-flooded zone, phosphate is taken up via photosynthesis limited by insolation and the present concentration of phosphate. This dependence is described by means of saturation functions (see Section~\ref{sub:saturation}). A fraction $\nu$ of the uptake is transformed into $DOP$, the remnants are exported into the deeper layer. Furthermore, $DOP$ is remineralized into $PO_4$ with a remineralization rate $\lambda$. Being independent of light this transformation takes place in both of the layers.
Altogether, these processes are represented by the nonlinear coupling term 
\begin{displaymath}
d:L^2(\Om\times[0,T])^2\to L^2(\Om\times[0,T])^2 \mbox{ with }d:=(d_1,d_2) \mbox{ defined by }
\end{displaymath}
\begin{displaymath}
d_1(y,x,t):=
\begin{cases}
 -\lambda y_2(x,t)+G(y_1,x,t)&\mbox{in $\Om_1\times[0,T]$,}\\
 -\lambda y_2(x,t)+\bar{F}(y_1,x,t)&\mbox{in $\Om_2\times[0,T]$}
\end{cases}
\end{displaymath}
and
\begin{displaymath}
d_2(y,x,t):=
\begin{cases}
 \lambda y_2(x,t)-\nu G(y_1,x,t)&\mbox{in $\Om_1\times[0,T]$,}\\
 \lambda y_2(x,t)&\mbox{in $\Om_2\times[0,T].$}
\end{cases}
\end{displaymath}
In detail, we write
\begin{align*}
G(y_1,x,t)&:=\alpha\frac{y_1(x,t)}{|y_1(x,t)|+K_P}\frac{I(x^{\prime\!\!\!},t)\text{e}^{-x_3K_W}}{|I(x^{\prime\!\!\!},t)\text{e}^{-x_3K_W}|+K_I} &&\mbox{ on } \Om_1\times[0,T]
\end{align*}
%\\
for the biological uptake of phosphate. The maximum rate $\alpha$ is limited by  the present concentration of phosphate and insolation according to Michaelis-Menten kinetics. Insolation is represented by the bounded function $I\in L^{\infty}(\Om^{\prime}\times[0,T])$ depending on time and the water surface. The export function
\begin{align*}
E(y_1,x^{\prime},t)&:=(1-\nu)\int_0^{h_e(x^{\prime})}G(y,(x^{\prime\!\!\!},x_3),t)dx_3 &&\mbox{ on } \Om^{\prime}\times[0,T]
\end{align*}
stands for the fraction of consumed phosphorus exported into the deeper layer. The integral over the depth of the euphotic zone ensures that the material in the whole water column is exported at the same time. Finally,
%\\
\begin{align*}
\bar{F}(y_1,x,t)&:=-E(y_1,x^{\prime},t)\frac{\beta}{\bar{h}_e}\left(\frac{x_3}{\bar{h}_e}\right)^{-\beta-1} &&\mbox{ on } \Om_2\times[0,T]
\end{align*}
%each for all $y\in L^2(0,T;L^2(\Om))^2$. 
represents the export being reduced while sinking through the second layer. The reduction is achieved by multiplication of a factor smaller than one.

The model parameters are assembled in the vector $(\lambda,\alpha,K_{P},K_{I},K_{W},\beta,\nu)\in\mathbb{R}^7$. In the cited paper, their values are determined via laboratory experiments or observations although the authors point out that some of them are not well known and maybe not even constant.
\subsection{Boundary conditions and initial value}
To obtain a mathematically well-posed problem, we will impose restrictions about the tracers' behavior on the boundary and at the initial time $t=0$.

Usually, the original formulation of an ecosystem model does not provide explicit statements about the behavior on the boundary. This is also true for the $PO_4$-$DOP$-model. However, since there are no sources or sinks it is appropriate to demand that the total amount of tracer concentrations in $\Om$ is constant. Neglecting molecular diffusion, i.e. $\kappa_1=\kappa_2$, the corresponding condition $\frac{d}{dt}\int_{\Om} (y_1+y_2)(x,t)dx=0$ is equivalent to the Neumann boundary condition
\begin{displaymath}
 \nabla y_j(x,t)\cdot(\kappa(x,t)\eta(x))+ b_j(y,x,t)=0
\end{displaymath}
for all $(x,t)\in\Ga\times[0,T]$ and $j\in\{1,2\}$ with the non-local coupling term
\begin{displaymath}
b:L^2(\Om\times[0,T])^2\to L^2(\Ga\times[0,T])^2 \mbox{ with }
b:=(b_1,b_2) \mbox{ defined by }
\end{displaymath}
\begin{displaymath}
b_1(y,x,t):=
\begin{cases}
 -E(y_1,x^{\prime},t)&\mbox{for $x=(x^{\prime},x_3)\in\Ga_1$, $t\in[0,T]$,}\\
 -E(y_1,x^{\prime},t)\left(\frac{x_3}{\bar{h}_e}\right)^{-\beta}&\mbox{for $x=(x^{\prime},x_3)\in\Gamma_2$, $t\in[0,T]$,}\\
0&\mbox{for $x=(x^{\prime},0)\in\Gamma^{\prime}$, $t\in[0,T]$}
\end{cases}
\end{displaymath}
and $b_2(y,x,t)=0$. Neumann boundary conditions are natural for problems given by transport equations. They specify the derivative alongside the vector $\kappa(s,t)\eta(s)$, where $\eta(s)$ is the outward pointing unit normal vector, and thereby reflect the change of tracer concentration at the boundary. The fact $b_1\neq0$ signifies that phosphate escapes through the boundary. Since there is no aphotic zone beneath $\Ga_1$, here, the total export $E$ leaves $\Om$. The export escaping through $\Ga_2$ is reduced according to the depth of the superjacent aphotic zone. With respect to $DOP$, the model is designed as a closed system, i.e. the total $DOP$ concentration is remineralized into phosphate. This corresponds to $b_2=0$.

We additionally fix an initial value $y_{0}^j$ which is a function of $\Om$ for $j=1,2$ representing the tracers' concentration at $t=0$:
\begin{displaymath}
y_j(x,0)=y_{0}^j(x).
\end{displaymath}
The vector of all initial values will be denoted by $y_0:=(y_{0}^1,y_0^2)$.

In total, we arrive at the initial boundary value problem
 \begin{align}
\label{eq:zustand}
%\left.
\begin{array}{rcll}
\displaystyle \partial_ty_j(x,t)+\vec v(x,t)\cdot\nabla y_j(x,t)-{\rm div}(\kappa_j(x,t)\nabla  y_j(x,t))+d_j(y,x,t)&=&0
\\
\displaystyle \nabla y_j(x,t)\cdot(\kappa(x,t)\eta(x))+ b_j(y,x,t)&=&0
\\
y_j(x,0)&=&y_{0}^j(x)
\end{array}
%\right\}\quad
\end{align}
$\text{for all }
j=1,2.$ To compute the derivative of $y$ solving \eqref{eq:zustand} with respect to the parameters we need the solution $h$ of the linearized equation  
 \begin{align}\label{eq:zustand_derivative}
%\!\!\!\left.
\begin{array}{rcll}
\displaystyle \partial_th_j(x,t)+\vec v(x,t)\cdot\nabla h_j(x,t)-{\rm div}(\kappa_j(x,t)\nabla  h_j(x,t))+\partial_yd_j(y)h(x,t)&=&f(x,t)
\\
\displaystyle \nabla h_j(x,t)\cdot(\kappa(x,t)\eta(x))+ \partial_yb_j(y)h(x,t)&=&g(x,t)
\\
h_j(x,0)&=&0
\end{array}%\right\}
\end{align}
$\text{for all }
j=1,2.$ Here, $f$ and $g$ denote the derivatives of $d$ and $b$ with respect to the parameters. 

\section{Mathematical formulation}\label{sec:weak}
The analogous structure of the systems \eqref{eq:zustand} and \eqref{eq:zustand_derivative} suggests to carry out the mathematical analysis for a generalized initial boundary value problem. To cover also a variety of other models we will consider an arbitrary number of equations with unspecified reaction terms and the dimension $n\leq3$ for $\Om$. The $NPZD$-model, presented by R\"uckelt et al. \cite{rue10}, for instance, is defined on a one-dimensional water column. Having analyzed the generalized problem we will specialize the results with respect to the $PO_4$-$DOP$-model.
\subsection{General assumptions}\label{sec:generalassumptions}
Throughout this paper, let $s\in\mathbb{N}$, $n\leq3$, $T>0$ and $\Omega\subset\mathbb{R}^n$ be an open, bounded set with a Lipschitz boundary\footnote{For a definition see e.g. Tr\"oltzsch \cite[Section~2.2]{tr02}.} $\Gamma:=\partial\Omega$. $\eta(s)$ denotes the outward-pointing unit normal vector in $s\in\Ga$. We abbreviate $Q_T:=\Omega\times(0,T)$ and $\Sigma:=\Gamma\times(0,T)$.

Consider further $\vec v\in L^\infty(0,T; H^1(\Om)^n)$ with the properties ${\rm div}(\vec v(t))=0$ in $L^2(\Om)$ and $\vec v(t)\cdot\eta=0$ in $L^2(\Ga)$, each for almost every $t\in [0,T]$. Let $\kappa\in L^\infty(Q_T)^s$ with
$\kappa_{\min}:=\essinf\{\kappa_j(x,t);(x,t)\in Q_T,j=1,\dotsc,s\}>0$ and $y_0\in L^2(\Om)^s$. We denote
 $\vec v_{\max}:=\|\vec v\|_{L^\infty(0,T;H^1(\Om)^n)}$ and $\kappa_{\max}:=\esssup\{\kappa_j(x,t);(x,t)\in Q_T,j=1,\dotsc,s\}$.
 
Finally, we consider the continuous reaction terms 
\begin{displaymath}
d:L^2(Q_T)^s\to L^2(Q_T)^s\text{ and } b:L^2(Q_T)^s\to L^2(\Sigma)^s
\end{displaymath}
defined by the indexed families $(d(t))_t$ and $(b(t))_t$ of operators
\begin{displaymath}
d(t):L^2(\Om)^s\to L^2(\Om)^s\text{ and } b(t):L^2(\Om)^s\to L^2(\Ga)^s
\end{displaymath}
via $d(y,x,t):=d(y)(x,t):=d(t)(y(t))(x)$ and $b(y,x,t):=b(y)(x,t):=b(t)(y(t))(x)$.

\subsection{Notation and preliminaries}
The mathematical investigations in this paper are based on the theories of normed linear spaces (especially of $L^p$-functions) and Hilbert spaces \cite{con97,rud87}. 

Throughout the paper, norms will usually be distinguished by an index indicating the corresponding space. An exception is made for the Hilbert space $L^2(E)^s$ of $s$-dimensional vectors of quadratically integrable functions on a set $E$. Here, we write $\|.\|_{E^s}$ instead of $\|.\|_{L^2(E)^s}$. If $s=1$ the index $s$ is omitted. The same rule applies for inner products in Hilbert spaces being generally defined by round brackets $(.\,,.)$ with the corresponding index. In contrast, the scalar product in $\mathbb{R}^n$ is denoted by a dot.

The applications of linear functionals (dual pairings) are denoted by angle brackets $\langle.\,,.\rangle$ subscripted by the corresponding dual space. Dual pairings without any index belong to the space $(H^1(\Om)^*)^s$ and are defined by
\begin{displaymath}
\langle f,v\rangle:=\langle f_1,v_1\rangle_{H^1(\Om)^*}+\dots+\langle f_s,v_s\rangle_{H^1(\Om)^*}\quad\text{for all $f\in (H^1(\Om)^*)^s$ and }v\in H^1(\Om)^s.
\end{displaymath}

Similarly, given a Hilbert space $H$, the inner product on the Cartesian product $H^s$ is defined by
\begin{displaymath}
(x,y)_{H^s}:=(x_1,y_1)_{H}+\dots+(x_s,y_s)_{H}\quad\text{for all $x,y\in H$}.
\end{displaymath}
The product Hilbert space is always endowed with the norm induced by this inner product.

Functions in two variables (on $Q_T$ or $\Sigma$) are usually regarded as abstract functions defined on $[0,T]$ with values in a function space on $\Om$ or $\Ga$, respectively. An introduction to these functions is given e.g. by Gajewski et al. \cite{gaj67}. In the context of time-dependent partial differential equations, the abstract function space
\begin{displaymath}
W(0,T):=\{y\in L^2(0,T; H^1(\Om));\,y^{\prime}\in L^2(0,T; H^1(\Om)^*)\}
\end{displaymath}
is of great significance.
The weak derivative $y^{\prime}$ is called \emph{distributional} since it is no function. The formal definition is given e.g. by R\r{u}\v{z}i\v{c}ka \cite{ru10}. The space $W(0,T)$ is well investigated. Some important properties are summarized in the theorem beneath. The proofs of the first two statements are extensions of the results in Sec.~9.3. of Evans \cite{evans}. The third statement is a special case of Theorem IV.1.17 by Gajewski et al. \cite{gaj67}.
\begin{theorem}\label{satz:hauptsatz} The following properties are valid:
\begin{enumerate}
  \item The space $W(0,T)$ is continuously embedded in $C([0,T];L^2(\Om))$, i.e. there is a constant $C_E>0$ with
\begin{equation*}
 \|y\|_{C([0,T];L^2(\Om))}\leq C_E\|y\|_{W(0,T)}\quad\text{for all $y\in W(0,T)$}.
\end{equation*}
  \item For each $y\in W(0,T)$ the map $t\mapsto\|y(t)\|_{L^2(\Om)}^2$ is weakly differentiable with the almost everywhere defined weak derivative $\frac{d}{dt} \|y(t)\|_{L^2(\Om)}^2=2\langle y^{\prime}(t),y(t)\rangle_{H^1(\Omega)^*}$.
  \item For all 
  $y\in W(0,T)$ the following ``fundamental theorem'' holds:
\begin{displaymath}
\int_0^T\langle y^{\prime}(t),y(t)\rangle_{H^1(\Omega)^*}dt=\frac{1}{2}(\|y(T)\|_{L^2(\Om)}^2-\|y(0)\|_{L^2(\Om)}^2).
\end{displaymath}
\end{enumerate}
\end{theorem}
The next result provides a means to ``restrict'' elements of $H^1(\Om)$ to the boundary of $\Om$. The proof can be found in Evans \cite[Sec.~5.5]{evans}.
\begin{theorem}
\label{sa:spursatz}
 \textbf{\emph{(Trace Theorem)}}
 There is a linear and continuous map $\tau:H^1(\Om)\to L^2(\Gamma)$ that restricts continuous functions $y\in H^1(\Om)\cap C(\bar{\Om})$ to the boundary, i.e. $(\tau y)(x) =y(x)$ for all $x\in\Gamma$. The continuity of $\tau$ implies the existence of a constant $c_\tau>0$, depending solely on $\Om$, with the property $\|\tau y\|_{L^2(\Gamma)}\leq c_\tau\|y\|_{H^1(\Om)}$ for all $y\in H^1(\Om)$. 
\end{theorem}
\subsection{Weak formulation}
Initial boundary value problems like \eqref{eq:zustand} are usually solved in a weakened form, i.e. the requirements for the solution are relaxed. Also some numerical methods are designed to find weak solutions (cf. Galerkin's method in the proof of Thm.~\ref{satz:ex_allgemein}). For a one-dimensional initial boundary value problem, Tr\"oltzsch \cite{tr02} derives a weakened formulation that ensures that weak and classical solutions in $C^2(\bar{Q}_T)^s$ coincide as soon as the latter exist. In the following, we will adapt his argumentation to the $s$-dimensional system based on \eqref{eq:zustand} and on the assumptions of Sec.~\ref{sec:generalassumptions}. 

Let $w\in C^1(\bar{Q}_T)^s$ be a vector of test functions. As a first step, the original differential equation, evaluated in $(x,t)\in\Om\times[0,T]$, is multiplied by $w_j(x,t)$. By integrating with respect to $\Om$
 we obtain
 \begin{align*}
(\partial_t y_j(t),w_j(t))_{\Om}+(\vec v(t)\cdot\nabla y_j(t),w_j(t))_{\Om}-({\rm div}(&\kappa_j(t)\nabla  y_j(t)),w_j(t))_{\Om}\\
&+(d_j(y,.\,,t),w_j(t))_{\Om} = 0
\end{align*}
for every $j=1,\ldots,s$. In order to relax the requirements for $y_j$ the  
temporal derivative $y_j^{\prime}$ (regarded as an abstract function) is understood as a functional in $H^1(\Omega)^\ast$, i.e.
\begin{displaymath}
(\partial_t y_j(t),w_j(t))_{\Om}=\langle y_j^{\prime}(t),w_j(t)\rangle_{H^1(\Omega)^\ast}.
\end{displaymath}
The third summand is transformed by partial integration based on Gauss' divergence theorem. 
Inserting the boundary condition we obtain
\begin{align*}
-\int_\Omega{\rm div}(\kappa_j\nabla  y_j)w_jdx&= \int_\Omega(\kappa_j\nabla  y_j\cdot\nabla w_j)dx - \int_\Gamma(\nabla y_j\cdot(\kappa_j\eta))w_jds\\
&=\int_\Omega(\kappa_j\nabla  y_j\cdot\nabla w_j)dx + \int_\Gamma b_j(y,s,t)w_jds.
\end{align*}
In the integrands, we generally omitted the arguments $(x,t)$ and $(s,t)$, respectively.

All linear summands are subsumed under the time-dependent bilinear form $B:H^1(\Om)^s\times H^1(\Om)^s\times[0,T]\to\mathbb{R}$ given by $B(u,v;t):=\sum_{j=1}^sB_j(u_j,v_j;t)$ with components defined by
\begin{equation*}
 B_j(u_j,v_j;t):=\int_\Om(\kappa_j(t)\nabla u_j\cdot\nabla v_j)dx+  \int_\Om(\vec v(t)\cdot\nabla u_j)v_jdx.
 \end{equation*}
Later, we apply $B$ mostly to abstract functions $\alpha,\beta\in L^2(0,T;H^1(\Om))^s$ evaluated in a fixed $t$. In this case we will write $B(\alpha,\beta;t)$ instead of $B(\alpha(t),\beta(t);t)$. 

The previous steps lead to the weak formulation
\begin{align*}
 \langle y^{\prime}_j(t),w_j(t)\rangle_{H^1(\Omega)^\ast} + B_j(y_j,w_j;t) &+ (d_j(y,.\,,t),w_j(t))_{\Om}
 + (b_j(y,.\,,t),w_j(t))_\Gamma=0
 \end{align*}
for all $t\in[0,T]$ and all test functions. We obtain a weak formulation for the $s$-dimensional problem by integrating with respect to time and summing up the equations for $1,\dotsc,s$. 

The summands of $B$ are well-defined as long as $y,w\in L^2(0,T; H^1(\Om))^s$. The derivative with respect to time has to satisfy $y^{\prime}\in L^2(0,T; H^1(\Om)^*)^s$. Thus, $W(0,T)^s$ turns out to be an adequate solution space. 

Since $d(y,.\,,.)\in L^2(0,T; L^2(\Om))^s$ and $b(y,.\,,.)\in L^2(0,T; L^2(\Ga))^s$ the test function is required to be an element of $L^2(0,T; H^1(\Om))^s$. Thus, instead of a classical solution $y\in C^2(\bar{Q}_T)^s$ of the initial boundary value problem~\eqref{eq:zustand} we search for $y\in W(0,T)^s$ fulfilling
\begin{align}\label{eq:schwach_speziell}
\int_0^T\{\langle y^{\prime}(t),w(t)\rangle +B(y,w;t) +(d(y,.\,,t),w(t))_{\Om^s}
 + (b(y,.\,,t),w(t))_{\Gamma^s}\}dt=0
 \end{align}
for all test functions $w\in L^2(0,T;H^1(\Om))^s$ and the initial value condition $y(0)=y_0$.
Because of Thm.~\ref{satz:hauptsatz}(1) it is possible to evaluate the weak solution in $t=0$.

At the end of this section, we prove some important statements concerning the bilinear form $B$.
\begin{lemma} The following properties hold for all $y,v\in H^1(\Om)^s$ and almost all $t\in[0,T]$.
 \label{hi:bilevans}
\begin{enumerate}
  \item There is a constant $C_B>0$ independent of $t,y,v$ such that 
\begin{equation*}
 |B(y,v;t)|\leq C_B\|y\|_{H^1(\Om)^s}\|v\|_{H^1(\Om)^s}.
\end{equation*}
  \item 
$\kappa_{\min}\|y\|^2_{H^1(\Om)^s}\leq B(y,y;t)+\kappa_{\min}\|y\|^2_{L^2(\Om)^s}$
  \item $B$ is monotone, i.e. 
\begin{equation*}
 B(y,y-v;t)-B(v,y-v;t)\geq0.
\end{equation*}
\end{enumerate}
\end{lemma}
\begin{proof}Since $B$ is defined by a sum of $s$ analogous components it suffices to confine the proof to the case $s=1$. 

Let $y,v\in H^1(\Om)$ and $t\in[0,T]$. In order to prove the first statement we obtain by means of the Cauchy-Schwarz inequality in $L^2(\Om)^n$
\begin{align*}
&|\int_\Om(\kappa(t)\nabla y\cdot\nabla v)dx|\leq
\kappa_{\max}|(\nabla y,\nabla v)_{\Om^n}|\leq\kappa_{\max}\|\nabla y\|_{\Om^n}\|\nabla v\|_{\Om^n}\quad\text{and}\\
& |\int_\Om(\vec v(t)\cdot\nabla y)vdx|=|(\vec v(t)v,\nabla y)_{\Om^n}| \leq \|\vec v(t)v\|_{\Om^n}\|\nabla y\|_{\Om^n}
\end{align*}
for all $t$ outside of some measure-zero set. For every $i=1\dotsc,n$, H\"older's inequality with the exponents $p=\frac{3}{2}$ and $q=3$ provides
\begin{align*}
\|\vec v_i(t)v\|_{\Om}=(\int_\Om \vec v_i(t)^2v^2dx)^\frac{1}{2}&
\leq(\int_\Om \vec v_i(t)^3dx)^\frac{1}{3}(\int_\Om v^6dx)^\frac{1}{6}
= \|\vec v_i(t)\|_{L^3(\Om)}\|v\|_{L^6(\Om)}
\end{align*}
and therefore
\begin{displaymath}
\|\vec v(t)v\|_{\Om^n}=(\sum\limits_{i=1}^n\|\vec v_i(t)v\|_{\Om}^2)^{\frac{1}{2}}\leq (\sum\limits_{i=1}^n\|\vec v_i(t)\|_{L^3(\Om)}^2\|v\|_{L^6(\Om)}^2)^{\frac{1}{2}}= 
\|\vec v(t)\|_{L^3(\Om)^n}\|v\|_{L^6(\Om)}.
\end{displaymath}
For each $r\in\{3,6\}$, there is a constant $c_{r}>0$ with $\|w\|_{L^r(\Om)}\leq c_{r}\|w\|_{H^1(\Om)}$ for all $w\in H^1(\Om)$ because of the continuous embedding $H^1(\Om)\hookrightarrow L^r(\Om)$. Taking into account the definition of the norm in $H^1(\Om)$ this leads to
\begin{align*}
|\!\!\int_\Om\!\!(\vec v(t)\cdot\nabla y)vdx|&\leq c_3\|\vec v(t)\|_{H^1(\Om)^n}c_{6}\|v\|_{H^1(\Om)}\|\nabla y\|_{\Om^n}\!\leq \!\vec v_{\max}c_{3}c_6\|v\|_{H^1(\Om)}\|y\|_{H^1(\Om)}
\end{align*}
provided that $t$ does not belong to a certain measure-zero set. Combining the results we obtain
\begin{displaymath}
|B(y,v;t)|\leq C_B\|y\|_{H^1(\Om)}\|v\|_{H^1(\Om)}
\end{displaymath}
for almost all $t\in[0,T]$ with the constant $C_B:=\kappa_{\max}+\vec v_{\max}c_{3}c_6$ .

For a proof of the second statement we observe primarily that the second summand of $B(y,y;t)$ vanishes according to Lemma~\ref{hi:englischeslemma}, applied to $v:=\vec v(t)$ and $w:=y$. For almost every $t\in[0,T]$ we estimate the first summand by
\begin{displaymath}
\int_\Om(\kappa(t)\nabla  y\cdot\nabla y)dx\geq \kappa_{\min}(\nabla  y,\nabla y)_{\Omega^n}= \kappa_{\min}\|\nabla  y\|_{\Omega^n}^2.
\end{displaymath}
Therefore, we obtain $ \kappa_{\min}\|\nabla  y\|_{\Omega^n}^2\leq B(y,y;t)$. The assertion of the lemma follows from adding $\kappa_{\min}\|y\|_{\Om}^2$ on both sides of this inequality.

In order to prove (3) we apply (2) with $y-v$ instead of $y$. Subtracting $\kappa_{\min}\|y-v\|_{\Om}^2$ on both sides we obtain 
\begin{displaymath}
0\leq\kappa_{\min}\|\nabla (y-v)\|^2_{\Om^n}\leq B(y-v,y-v;t).
\end{displaymath}
This corresponds to the assertion of the lemma since $B$ is bilinear.
\end{proof}
At last, we add an auxiliary lemma for the proof above.
\begin{lemma}
 \label{hi:englischeslemma}
 Let $v\in H^1(\Om)^n$ with ${\rm div}v=0$ in $H^1(\Om)$ and $v\cdot\eta=0$ in $L^2(\Gamma)$. Hence 
 \begin{displaymath}
\int_\Om (v\cdot\nabla w)wdx=0\quad\text{for all $w\in H^1(\Om)$.}
\end{displaymath}
\end{lemma}
\begin{proof}
Let $w \in H^1(\Om)$. For all $c\in H^1(\Om)$ und $x\in H^1(\Om)^n$ we prove the product rule
\begin{align*}
 {\rm div}(cx)=\!\sum\limits_{i=1}^n\partial_i(cx_i)=\!\sum\limits_{i=1}^n(c(\partial_ix_i)+x_i(\partial_ic))
 =c\sum\limits_{i=1}^n\partial_ix_i+\sum\limits_{i=1}^nx_i(\partial_ic)
=c\,{\rm div}x+x\cdot\nabla c.
\end{align*}
On the one hand, applying this formula to $c:=w$ and $x:=v$ we obtain
 \begin{align*}\label{eq:engllemfuerlemevans}
\int_\Om\!\!{\rm div}(vw)wdx=\!\int_\Om\!\! w^2{\rm div}vdx+\!\int_\Om \!\!(v\cdot\nabla w)wdx=\!\int_\Om \!\!(v\cdot\nabla w)wdx.
\end{align*}
The summand with $w^2$ vanishes because of the assumption about the divergence of $v$. On the other hand, the same formula applied to $c:=w$ and $x:=vw$ yields an integral over the divergence of $(vw^2)$ which can be transformed into a boundary integral by virtue of Gauss' divergence theorem. We obtain
\begin{align*}
\int_\Om{\rm div}(vw)wdx&=
\int_\Om{\rm div} (vw^2)dx-\int_\Om ((vw)\cdot\nabla w)dx%=\int_\Ga((vw^2)\cdot\eta) ds-\int_\Om ((vw)\cdot\nabla w)dx
\\
&=\int_\Ga(v\cdot\eta)w^2ds-\int_\Om (v\cdot\nabla w)wdx
=
-\int_\Om (v\cdot\nabla w)wdx.
\end{align*}
The first integral in the second line vanishes because the product of $v$ with the outward-pointing normal $\eta$ is assumed to be zero. Subtracting both of the results we arrive at 
\begin{displaymath}
2\int_\Om(v\cdot\nabla w)wdx=\int_\Om {\rm div}(vw)wdx-\int_\Om {\rm div}(vw)wdx=0,
\end{displaymath}
the statement of the lemma.
\end{proof}
%
%
%
%GALERKIN
\section{An existence and uniqueness result with Galerkin's method}\label{sec:ex_galerkin}
In the following, we investigate the unique solvability of the initial value problem in Eq.~\eqref{eq:schwach_speziell}. Problems of this kind were treated in literature with different methods according to the assumptions about the reaction terms. Banach's Fixed Point Theorem is used to solve nonlinear, Lipschitz continuous problems \cite{evans,zei85}. Galerkin's method is applied to monotone or linear reaction terms \cite{la68,tr02,gaj67} and to pseudo-monotone and coercive operators \cite{ru10}. Raymond et al. \cite{ray98} follow an alternative approach assuming a boundedness condition from below. However, this condition, just like the coercivity, seldom applies to the specific reaction terms of marine ecosystem models. 

Due to the frequent appearance of monotone and Lipschitz continuous reaction terms in actual models (cf. Sec.~\ref{sub:saturation}) the above-mentioned methods by Galerkin and Banach seem most adequate for their investigation. We will follow both approaches since, as we will see below, each of them has its individual benefits. 

In either case, we extend the respective standard proof from literature with the objective of allowing an arbitrary number of model equations as well as reaction terms with both Lipschitz continuous and monotone parts.

In the following, we state a first existence and uniqueness theorem and give its proof by means of Galerkin approximation. The preceding proposition states important estimates for weak solutions and thus contributes to the proofs in both the current and the following section. The latter contains a second existence and uniqueness result proved by Banach's Fixed Point Theorem. 

To increase the range of application of our results we investigate a generalization of Eq.~\eqref{eq:schwach_speziell}. To this end, we need the following assumptions.

Let $Y\hookrightarrow L^2(0,T;L^2(\Om))^s$ be a Banach space with $W(0,T)^s\hookrightarrow Y$.  For $i \in\{1,2\}$, we assume the operators $F_i:Y\to L^2(0,T;H^1(\Om)^*)^s$ to be generated by an indexed family of operators $F_i(t):H^1(\Om)^s\to (H^1(\Om)^*)^s$ for almost every $t\in[0,T]$, i.e.
\begin{align}\label{eq:generating}
\langle F_i(y),v\rangle_{L^2(0,T;H^1(\Om)^*)^s}=\int_0^T\!\!\langle F_i(y(t)),v(t)\rangle dt
\end{align}
for all $y\in Y,v\in L^2(0,T;H^1(\Om))^s$. Shortly, we wrote $F_i(y(t))$ instead of $F_i(t)(y(t))$ as we will do throughout this section. Let furthermore $f\in L^2(0,T;H^1(\Om)^*)^s$.

In the following, we search for a solution $y\in W(0,T)^s$ of the initial value problem  
\begin{align}\label{eq:schwach_allg}
y^{\prime} +\int_0^T B(y,.\,;t)dt + F_1(y)+ F_2(y)&=f\\
y(0)&=y_0.\nonumber
 \end{align}
 Due to the assumptions, elements of $W(0,T)^s$ belong to the domain of $F_1$ and $F_2$. The next lemma
 will explain in which way Eq.~\eqref{eq:schwach_allg} generalizes the initial value problem in Eq.~\eqref{eq:schwach_speziell}. In particular, we show in which way the original reaction terms $d$ and $b$ can be identified with the abstract operators $F_i:Y\to L^2(0,T;H^1(\Om)^*)^s$.
\begin{lemma}\label{lem:db_erz_funkt} %Let the assumptions of Thm.~\ref{satz:ex_speziell} be valid. 
Let $Y:=L^2(0,T;L^2(\Om))^s$. The operators $\tilde{d},\tilde{b}:Y\to L^2(0,T;H^1(\Om)^*)^s$ with
\begin{displaymath}
\tilde{d}(y):w\mapsto\int_0^T(d(y,.\,,t),w(t))_{\Om^s}dt\quad\text{ and } \quad \tilde{b}(y):w\mapsto\int_0^T(b(y,.\,,t),\tau w(t))_{\Ga^s}dt
\end{displaymath}
for all $y\in Y$ are well-defined and generated (in the sense of Eq.~\eqref{eq:generating}) by the indexed families $(\tilde{d}(t))_t$ and $(\tilde{b}(t))_t$ of operators $\tilde{d}(t),\tilde{b}(t):H^1(\Om)^s\to (H^1(\Om)^*)^s$ with
 \begin{displaymath}
\tilde{d}(t)(v):z\mapsto(d(t)(v),z)_{\Om^s}\quad\text{ and } \quad \tilde{b}(t)(v):z\mapsto(b(t)(v),\tau z)_{\Ga^s}.
\end{displaymath}
Here, $\tau$ is the map of Thm.~\ref{sa:spursatz}.
\end{lemma}
\begin{proof} The operators $\tilde{d},\tilde{b}$ are generated by $(\tilde{d}(t))_t,(\tilde{b}(t))_t$ because of the assumptions of Sec~\ref{sec:generalassumptions} about $d,b$ being defined by $(d(t))_t,(b(t))_t$. Thus, it remains to be shown that the operators are well-defined. First, $H^1(\Om)^s$ is an admissible domain of definition for the generating operators because it is a subspace of $L^2(\Om)^s$. Further, assuming $v\in H^1(\Om)^s$, we have to show that $\tilde{d}(t)(v)$ and $\tilde{b}(t)(v)$ are elements of $(H^1(\Om)^*)^s$. It suffices to investigate their boundedness since they are both obviously linear. Assuming $z\in H^1(\Om)^s$, we conclude with the Cauchy-Schwarz inequality in $L^2(\Ga)^s$ and the Trace Theorem~\ref{sa:spursatz} 
\begin{align*}
\langle \tilde{b}(t)(v),z\rangle=(b(t)(v),\tau z)_{\Ga^s}\leq\|b(t)(v)\|_{\Ga^s}\|\tau z\|_{\Ga^s}
\leq\|b(t)(v)\|_{\Ga^s}c_{\tau}\|z\|_{H^1(\Om)^s}.
\end{align*}
Thus, $\|\tilde{b}(t)(v)\|_{(H^1(\Om)^*)^s}\leq\|b(t)(v)\|_{\Ga^s}c_{\tau}$. The upper bound is finite since $b(t)(v)\in L^2(\Ga)^s$. Similarly, we obtain $\|\tilde{d}(t)(v)\|_{(H^1(\Om)^*)^s}\leq\|d(t)(v)\|_{\Om^s}<\infty$. 

The functionals on the spaces involving time are integrals over the generating functionals. Because of the identification $(L^2(0,T; H^1(\Om))^s)^*\cong L^2(0,T; H^1(\Om)^*)^s$ \cite[Thm.~IV.1.14]{gaj67} it suffices to show that the norm of the generating functionals is quadratically integrable. For $y\in Y$, we estimate
\begin{align*}
\|\tilde{b}(y)\|^2_{L^2(0,T;H^1(\Om)^*)^s}=
\!\int\limits_0^T\!\!\|\tilde{b}(t)(y(t))\|_{(H^1(\Om)^*)^s}^2dt
\leq\!\int\limits_0^T\!\!c_{\tau}\|b(t)(y(t))\|_{\Ga^s}^2dt=
c_{\tau}\|b(y)\|_{\Sigma^s}^2
\end{align*}
and the last expression is finite due to the definition of $b$. An analogous result follows for $\tilde{d}$. Thus, $\tilde{d}(y),\tilde{b}(y)$ are elements of $L^2(0,T;H^1(\Om)^*)^s$.
\end{proof}

%
% GALERKIN Theorem
%
\begin{theorem}\label{satz:ex_allgemein} 
Let the operators $F_i$ be continuous and fulfill the homogeneity condition $F_i(0)=0$. We assume that $F_2$ is monotone, i.e.
\begin{displaymath}
\langle F_2(y(t))-F_2(v(t)),y(t)-v(t)\rangle\geq0,
\end{displaymath}
and that there is a constant $L_2>0$, not depending on $t$, with (see also Appendix A)
\begin{displaymath}
 \|F_2(y(t))\|_{(H^1(\Om)^*)^s}\leq L_2\|y(t)\|_{H^1(\Om)^s},
\end{displaymath}
each for all $y,v\in Y\cap L^2(0,T;H^1(\Om))^s$ and almost all $t\in[0,T]$. Let further $F_1$ fulfill the Lipschitz condition
\begin{displaymath}
\|F_1(y(t))-F_1(v(t))\|_{(H^1(\Om)^*)^s}\leq L_1\|y(t)-v(t)\|_{\Om^s}
\end{displaymath}
for all $y,v\in Y$ and almost all $t\in[0,T]$ with $L_1>0$ independent of $t$.
Moreover, we assume either the embedding $W(0,T)^s\hookrightarrow Y$ to be compact  or one of the conditions
\begin{enumerate}
\item $F_2\neq0$, $L^2(0,T;H^1(\Om))^s\hookrightarrow Y$ and $F_1$ is strongly continuous. 
\item $F_2=0$ and $F_1$ is weakly continuous. 
\end{enumerate}
Then the initial value problem \eqref{eq:schwach_allg} has a unique weak solution $y\in W(0,T)^s$. 
\end{theorem}
As announced above, the proof of Theorem~\ref{satz:ex_allgemein} will follow after a proposition about estimates of weak solutions.
%%%%%%%%%%Energy proposition
\begin{proposition}\label{prop:energy_estimates}
Let $F_i:Y\to L^2(0,T; H^1(\Om)^*)^s$ be operators fulfilling the assumptions of Theorem~\ref{satz:ex_allgemein} and $Z\subseteq H^1(\Om)$ be a closed subspace. Let further $z_1,z_2$ be elements of $W(0,T)^s$ in case $Z= H^1(\Om)$ or else of $H^1(0,T;Z)^s$. Consider the difference $z:=z_1-z_2$ to fulfill
\begin{align}\label{eq:prop_eq}
\langle z^{\prime}(t),v\rangle+B(z,v;t) + \sum_{i=1}^{2} \langle F_{i}(z_1(t))-F_{i}(z_2(t)),v\rangle= \langle f(t),v\rangle
 \end{align}
for all $v\in Z^s$ and almost every $t\in [0,T]$. Hereby, we define $\langle z^{\prime}(t),v\rangle:=(z^{\prime}(t),v)_{\Om^s}$ if $z^{\prime}$ is a function.
Then the estimate
$$\|z\|_{C(\left[0,T\right],L^2(\Omega))^s}+
\|z\|_{L^2(0,T;H^1(\Om))^s}\leq C(\|f\|_{L^2(0,T;H^1(\Om)^*)^s}+\|z(0)\|_{L^2(\Omega)^s})$$
holds with a constant $C>0$ independent of $z, f, F_2$.

In case $z_2=0$ there is another constant $\tilde{C}>0$, independent of $z, f$, with
\begin{displaymath}
\|z^{\prime}\|_{L^2(0,T;H^1(\Om)^*)^s}\leq \tilde{C}(\|f\|_{L^2(0,T;H^1(\Om)^*)^s}+\|z(0)\|_{L^2(\Omega)^s}).
\end{displaymath}
\end{proposition}
\begin{proof} Since $z(t)\in Z^s$ for almost all $t \in [0,T]$ Eq.~\eqref{eq:prop_eq} implies in particular
\begin{eqnarray}\label{eq:prop_eq_en_eingesetzt}
\langle z^{\prime}(t),z(t)\rangle+B(z,z;t) + \sum_{i=1}^{2} \langle F_{i}(z_1(t))-F_{i}(z_2(t)),z(t)\rangle = \langle f(t),z(t)\rangle
 \end{eqnarray}
for these $t$.
First we observe
\begin{equation*}
\langle F_{2}(z_1(t))-F_2(z_2(t)),z(t)\rangle=\langle F_{2}(z_1(t))-F_2(z_2(t)),z_1(t)-z_2(t)\rangle\geq 0
\end{equation*}
by the monotonicity condition assumed for $F_2$. Using additionally Theorem~\ref{satz:hauptsatz}(2), Eq.~\eqref{eq:prop_eq_en_eingesetzt} leads to the estimate
\begin{eqnarray*}
\frac{1}{2}\frac{d}{dt}\|z(t)\|_{\Om^s}^2+B(z,z;t)\leq\langle f(t),z(t)\rangle-\langle F_{1}(z_1(t))-F_{1}(z_2(t)),z(t)\rangle.
\end{eqnarray*}
Both of the summands on the right-hand side are estimated by means of their boundedness and afterwards by Cauchy's inequality with an arbitrary $\varepsilon>0$ \cite[Appendix B.2]{evans}. This provides for the first summand 
\begin{displaymath}
\langle f(t),z(t)\rangle\leq
\|f(t)\|_{(H^1(\Om)^*)^s}\|z(t)\|_{H^1(\Om)^s}\leq
\frac{1}{4\varepsilon}\|f(t)\|_{(H^1(\Om)^*)^s}^2+\varepsilon\|z(t)\|_{H^1(\Om)^s}^2.
\end{displaymath}
Employing additionally the Lipschitz condition we obtain for the second summand  
\begin{align*}
|\langle F_{1}(z_1(t))-&F_{1}(z_2(t)),z(t)\rangle |\leq
\|F_{1}(z_1(t))-F_{1}(z_2(t))\|_{(H^1(\Om)^*)^s}\|z(t)\|_{H^1(\Om)^s}\\
&\leq 
L_1\|z_1(t)-z_2(t)\|_{\Om^s} \|z(t)\|_{H^1(\Om)^s}
\leq 
\frac{L_1^2}{4\varepsilon}\|z(t)\|_{\Om^s}^2+\varepsilon \|z(t)\|_{H^1(\Om)^s}^2.
\end{align*}
Estimating the bilinear form $B$ according to Lemma~\ref{hi:bilevans}(2) we arrive at
\begin{align*}
 \frac{1}{2}\frac{d}{dt}\|z(t)\|_{\Om^s}^2
                                                        &\leq \frac{1}{4\varepsilon}\|f(t)\|_{(H^1(\Om)^*)^s}^2
                                                        +\varepsilon\|z(t)\|_{H^1(\Om)^s}^2
                                                        +\frac{L_1^2}{4\varepsilon}\|z(t)\|_{\Om^s}^2\\
                                                        &+\varepsilon \|z(t)\|_{H^1(\Om)^s}^2+\kappa_{\min}\|z(t)\|_{\Om^s}^2-\kappa_{\min}\|z(t)\|^2_{H^1(\Omega)^s}.                                         
\end{align*}
By rearranging the summands and naming $c_1=2\kappa_{\min}+{L_1^2}/{2\varepsilon}$ the inequality is transformed into
\begin{align}
\label{eq:mith1norm}
\frac{d}{dt}\|z(t)\|^2_{\Om^s}&
\leq 
\frac{1}{2\varepsilon}\|f(t)\|^2_{(H^1(\Om)^*)^s}+c_1\|z(t)\|_{\Om^s}^2-2(\kappa_{\min}-2\varepsilon)\|z(t)\|^2_{H^1(\Omega)^s}\\  
                                  &
                                  \leq 
                                  \frac{1}{2\varepsilon}\|f(t)\|^2_{(H^1(\Om)^*)^s}+c_1\|z(t)\|_{\Om^s}^2\nonumber.
\end{align}
The last estimate holds for $\varepsilon<\kappa_{\min}/2$. This condition implies $\kappa_{\min}-2\varepsilon>0$ and thus the negativity of the last summand.

The well-known lemma of Gronwall \cite[Appendix B.2]{evans} yields
\begin{align}\label{eq:zitatbewbanach}
\|z(t)\|_{\Om^s}^2& \leq  \text{e}^{tc_1}\left[\|z(0)\|_{\Om^s}^2 + \int_0^t  \frac{1}{2\varepsilon}\|f(\sigma)\|^2_{(H^1(\Om)^*)^s} d\sigma\right]\nonumber\\
                    &\leq
                    C_1\left[\|z(0)\|_{\Om^s}^2+\|f\|^2_{L^2(0,T;H^1(\Om)^*)^s}\right]
\end{align}
for all $t\in[0,T]$ with $C_1:= \exp(Tc_1)\max\{1,1/(2\varepsilon)\}$. By regarding the supremum with respect to $t$ we obtain the boundedness in the norm of $C([0,T];L^2(\Omega))^s$.

For an analogous result in $L^2(0,T;H^1(\Omega))^s$ we return to Eq. (\ref{eq:mith1norm}). Choosing again $\varepsilon<\kappa_{\min}/2$ we add the negative summand with a positive sign to the other side of the inequality. Integrating with respect to $t$ we arrive at
\begin{eqnarray*}
\int_0^T\frac{d}{dt}\|z(t)\|_{\Om^s}^2dt + c_2\|z\|^2_{L^2(0,T;H^1(\Omega))^s} \leq \int_0^Tc_1\|z(t)\|_{\Om^s}^2dt +\frac{1}{2\varepsilon}\|f\|^2_{L^2(0,T;H^1(\Om)^*)^s}                       
\end{eqnarray*}
using the abbreviation $c_2:=2(\kappa_{\min}-2\varepsilon)>0$. The first integral is transformed by virtue of Theorem~\ref{satz:hauptsatz}. Due to the boundedness in the norm of $C([0,T];L^2(\Omega))^s$ the integrand on the right side is bounded with respect to $t$. We obtain
\begin{align*}
 \|z(T)\|_{\Om^s}^2\! \! + c_2\|z\|^2_{L^2(0,T;H^1(\Omega))^s}\!  \leq Tc_1\|z\|^2_{C([0,T];L^2(\Omega))^s} \! +\! \frac{1}{2\varepsilon}\|f\|^2_{L^2(0,T;H^1(\Om)^*)^s}\! +\! \|z(0)\|_{\Om^s}^2. 
\end{align*}
Since the summand $\|z(T)\|_{\Om^s}^2$ is nonnegative the estimate in $L^2(0,T;H^1(\Omega))^s$ follows from inserting the upper bound for $\|z\|^2_{C(\left[0,T\right];L^2(\Omega))^s}$. In summary we obtain 
\begin{eqnarray}\label{eq:abshc_in_h1}
\|z\|^2_{L^2(0,T;H^1(\Omega))^s}
                    \leq
                    C_2\left[\|z(0)\|_{\Om^s}^2+\|f\|^2_{L^2(0,T;H^1(\Om)^*)^s}\right]
\end{eqnarray}
with $C_2:={(C_1Tc_1+\max\{1/(2\varepsilon),1\})}/{2(\kappa_{\min}-2\varepsilon)}$.
The actual assertion follows from extracting the square root and estimate the right side by virtue of the binomial theorem. The estimation constant is given by $C:=\sqrt{C_1}+\sqrt{C_2}$.
%norm in the dual space 

To prove the boundedness of $z^{\prime}$ we assume $z_2=0$,  i.e. $z=z_1$, and choose $v\in H^1(\Om)^s$ with $\|v\|_{H^1(\Om)^s}=1$. In case $Z\neq H^1(\Om)$ and $Z$ is a closed subset, there are $v_1\in Z^s$, $v_2\in (Z^s)^{\bot}$ with $v=v_1+v_2$. Since $H^1(\Om)^s$ is dense in $L^2(\Om)^s$ the orthogonality of $v_2$ and $z^{\prime}(t)\in Z^s$ in $H^1(\Om)^s$ implies their orthogonality in $L^2(\Om)^s$. Thus, we conclude from Eq.~\eqref{eq:prop_eq}
\begin{displaymath}
\langle z^{\prime}(t),v\rangle\!  = (z^{\prime}(t),v)_{\Om^s} \!= (z^{\prime}(t),v_1)_{\Om^s} \!=  \langle f(t),v_1\rangle - B(z,v_1;t)
 - \!\sum_{i=1}^2\langle F_i(z(t))-F_i(0),v_1\rangle
\end{displaymath}
for almost every $t\in[0,T]$. In case $Z=H^1(\Om)$ this equation corresponds to \eqref{eq:prop_eq} since here $v=v_1$.

As above, the Lipschitz continuous summand is estimated by
\begin{displaymath}
\langle F_1(z(t))-F_1(0),v_1\rangle\leq  L_1\|z(t)\|_{\Om^s}\|v_1\|_{H^1(\Om)^s}.
\end{displaymath}
Due to the homogeneity and the boundedness condition we obtain for the second summand  
\begin{displaymath}
\langle F_2(z(t))-F_2(0),v_1\rangle \leq \|F_2(z(t))\|_{(H^1(\Om)^*)^s}\|v_1\|_{H^1(\Om)^s}
\leq  L_2\|z(t)\|_{H^1(\Om)^s}\|v_1\|_{H^1(\Om)^s}.
\end{displaymath}
Similarly, we treat the dual pairing given by $f(t)$. Additionally, Lemma~\ref{hi:bilevans}(1) is applied to the bilinear form $B$. Remark, that the orthogonality of $v_1,v_2$ and the Pythagorean theorem in the Hilbert space $H^1(\Om)^s$ imply 
\begin{displaymath}
\|v_1\|_{H^1(\Om)^s}^2\leq\|v_1\|_{H^1(\Om)^s}^2+\|v_2\|^2_{H^1(\Om)^s}=\|v\|_{H^1(\Om)^s}^2=1.
\end{displaymath} 
Finally, the norm in $L^2(\Om)^s$ is bounded by the norm in $H^1(\Om)^s$. We conclude
\begin{align*}
 \|z^{\prime}(t)\|_{(H^1(\Om)^*)^s}^2&:=\sup_{\|v\|_{H^1(\Om)^s}=1}\langle z^{\prime}(t),v\rangle^2
  \leq (\|f(t)\|_{(H^1(\Om)^*)^s}  + C_3\|z(t)\|_{H^1(\Om)^s})^2\\
  &\leq 2(\|f(t)\|_{(H^1(\Om)^*)^s}^2  + C_3^2\|z(t)\|_{H^1(\Om)^s}^2)
 \end{align*}
with $C_3:=C_B + L_1+L_2$. The last estimate is valid because of the convexity of the square function on $\mathbb{R}$. Taking into account \eqref{eq:abshc_in_h1} we arrive at 
\begin{align*}
\|z^{\prime}\|_{L^2(0,T;H^1(\Om)^*)^s}^2\! &=\!\! \int_0^T\!\!\|z^{\prime}(t)\|_{(H^1(\Om)^*)^s}^2
 dt
  \leq 2(\|f\|_{L^2(0,T;H^1(\Om)^*)^s}^2\! +C_3^2
  \|z\|_{L^2(0,T;H^1(\Om))^s}^2)\\
  &\leq C_4(\|z(0)\|^2_{\Om^s}+\|f\|_{L^2(0,T;H^1(\Om)^*)^s}^2).
\end{align*}
denoting $C_4:=2(1+C_3^2C_2)$. Again the assertion follows from extracting the square root. Thereby, the estimation constant is determined as $\tilde{C}:=\sqrt{C_4}$.
\end{proof}

%
%------------ GALERKIN PROOF---------------- 
%
%
%\subsection*{Proof of }
\begin{proof}[Proof of Thm. \ref{satz:ex_allgemein}]
This proof combines methods used by Tr\"oltzsch \cite{tr02} and Evans \cite{evans}.

We start choosing an orthogonal basis $(v_j)_{j\in\mathbb{N}}$ of the separable Hilbert space $H^1(\Omega)$. After a possible orthonor\-malization we can consider it to be an orthonormal basis of $L^2(\Omega)$ since $H^1(\Omega)$ is dense in this space. 

Let $l\leq s$. In the following, we will approximate the $l$-th component $y_l$ of the weak solution by a sequence $(y_{ln})_n$ whose $n$-th member belongs to the finite-dimensional subspace $\aufspann \{v_1, \dotsc,v_n\}$ of $H^1(\Omega)$. For this member we consider the ansatz
\begin{equation*}
y_{ln}(t)=\sum\limits_{i=1}^n{}^lu_i^n(t)v_i
\end{equation*}
at every point of time $t$ and, additionally, demand $y_{ln}(0)=y_0^l$. In the following we will determine the coefficients 
\begin{displaymath}
u^n:\left[0,T\right]\to\mathbb{R}^{n\times s}\quad \text{with}\quad u^n=
\left(
\begin{array}
{ccc}
{}^1u^n_1&\cdots& {}^su^n_1\\
\vdots&\ddots&\vdots\\
{}^1u^n_n&\cdots& {}^su^n_n
\end{array}
\right)
\end{displaymath}
such that $y_{ln}(t)$ solves
\begin{align}\label{eq:schwach_pw_basis}
(y_{ln}^{\prime}(t),v_j)_{\Om}+B_l(y_{ln},v_j;t) + \sum\limits_{m=1}^2\langle F_{ml}(y_{n}(t)),v_j\rangle_{H^1(\Om)^*} %+\langle F_{2l}(y_{n}&(t)),v_j\rangle_{H^1(\Om)^*}\nonumber\\
 = \langle f_l(t),v_j\rangle_{H^1(\Om)^*}
 \end{align}
for all $j\leq n$ and all $l\leq s$. 
Since $y_{ln}^{\prime}$ is a function the inner product $(y_{ln}^{\prime}(t),v_j)_{\Om}$ can be later perceived as a dual pairing in $H^1(\Om)^*$. 

Inserting the ansatz for $y_{ln}(t)$ into \eqref{eq:schwach_pw_basis} the linearity of the first summands and the orthonormality of the basis yield for the left side of the equation 
\begin{align*}
(\sum\limits_{i=1}^n{{}^lu_i^n}^{\prime}(t)v_i,v_j)_{\Om}+B(\sum\limits_{i=1}^n{}^lu_i^n(t)v_i,&v_j;t) 
+ \sum\limits_{m=1}^2\langle F_{ml}((\sum\limits_{i=1}^n{}^ku_i^n(t)v_i)_{k\leq s}),v_j\rangle_{H^1(\Om)^*}\\
% &\qquad\qquad+\langle F_{2l}((\sum\limits_{i=1}^n{}^ku_i^n(t)v_i)_{k\leq s}),v_j\rangle_{H^1(\Om)^*}\\
&=
{{}^lu_j^n}^{\prime}(t)+\sum\limits_{i=1}^n{}^lu_i^n(t)B(v_i,v_j;t) + \Phi_{jl}(t,u^n(t))
 \end{align*}
 where we combined the last two terms to a function of the coefficient matrix, namely
 \begin{displaymath}
\Phi_{jl}(t,u^n(t)):= \sum\limits_{m=1}^2\langle F_{ml}((\sum_{i=1}^n{}^ku_i^n(t)v_i)_{k\leq s}),v_j\rangle_{H^1(\Om)^*}. %+\langle F_{2l}((\sum_{i=1}^n{}^ku_i^n(t)v_i)_{k\leq s}),v_j\rangle_{H^1(\Om)^*}.
\end{displaymath}
The same arguments yield $(y_{0}^l,v_j)_{\Om}=(y_{ln}(0),v_j)_{\Om}=(\sum_{i=1}^n{}^lu_i^n(0)v_i,v_j)_{\Om}={}^lu_j^n(0)$ for the initial value.

Combining these equations  for all $j\leq n$ and $l\leq s$, we observe that the coefficient matrix $u^n$ solves the $(n\times s)$-dimensional nonlinear system of ordinary differential equations
\begin{align}\label{ode}
 \frac{d}{dt}u^n(t)+A(t)u^n(t)+\Phi(t,u^n(t))&=r(t)\\
 u^n(0)&=((y_{0}^l,v_j)_{\Om}) _{j=1,\dotsc,n,\atop l=1,\dotsc,s }.\nonumber
\end{align}
Here, we define the matrices
\begin{displaymath}
\Phi:=
(\Phi_{jl})_{j=1,\dotsc,n,\atop l=1,\dotsc,s }:\left[0,T\right]\times\mathbb{R}^{n\times s}\to\mathbb{R}^{n\times s},\quad
 r:=
 (\langle f_l(\,.\,),v_j\rangle)_{j=1,\dotsc,n,\atop l=1,\dotsc,s } :\left[0,T\right]\to\mathbb{R}^{n\times s}
\end{displaymath}
and $A:=(B(v_i,v_j;.\,))_{j=1,\dotsc,n,\atop i=1,\dotsc,n }\in L^2(0,T)^{n\times n}$.  In each case the index above counts the number of lines. 

The solvability of \eqref{ode} follows from the existence theorem of Carath\' eodory \cite[Thm.~2.1.1]{cod55}. Due to the assumed continuity of $F_i$ the operator $\Phi(t,.\,)$, being a composition of $F_i$ with continuous functions, is continuous with respect to $u^n$. Furthermore, the orthonormality of $(v_j)_{j\in\mathbb{N}}$ in $L^2(\Om)$ yields
\begin{align*}%\label{eq:coefficientenabschaetzung}
\|y_{n}(t)\|_{\Om^s}^2=\sum\limits_{l=1}^s\|\sum\limits_{i=1}^n{}^lu_i^n(t)v_i\|_{\Om}^2=
\sum\limits_{l=1}^s\sum\limits_{i=1}^n{}^lu_i^n(t)^2=
\|u^n(t)\|_{\mathbb{R}^{n\times s}}^2
\end{align*}
for the vector $y_n=(y_{1n},\dotsc,y_{sn})^{\top}$ whose elements are defined by the ansatz.
If the coefficient matrix $u^n$ solves problem \eqref{ode} the components of $y_n$ fulfill Eq.~\eqref{eq:schwach_pw_basis}. Thus, we can derive a priori estimates for $y_n$ and $u^n$ by means of Prop.~\ref{prop:energy_estimates} applied to the finite-dimensional and therefore closed subspace $Z:=\aufspann \{v_1,\dotsc,v_n\}$ of $H^1(\Omega)$. All elements of $Z$ are linear combinations of  $v_1, \dotsc,v_n$. Thus, the sum of the equations~\eqref{eq:schwach_pw_basis} for $j=1,\dotsc,n$ and $l=1,\dotsc,s$, each multiplied by an arbitrary constant,
corresponds to \eqref{eq:prop_eq} with $z_1=y_n$ and $z_2=0$ in $H^1(0,T;Z)^s$. Remark that the homogeneity conditions for $F_i$ allow to add $-F_i(0)$ to the corresponding summand.
 
Additionally, the definition via the ansatz leads to an estimate for $y_{n}(0)$. The associated proof uses the initial value of the coefficient matrix $u^n$, the orthonormality of the basis elements and Bessel's inequality. We conclude
\begin{align}\label{eq_anfangswert}
\|y_n(0)\|_{\Om^s}^2&=
\sum\limits_{l=1}^s\|\sum\limits_{i=1}^n{}^lu_i^n(0)v_i\|_{\Om}^2=\sum\limits_{l=1}^s\|\sum\limits_{i=1}^n(y_0^l,v_i)_{\Om}v_i\|_{\Om}^2=\sum\limits_{l=1}^s\sum\limits_{i=1}^n(y_0^l,v_i)_{\Om}^2\nonumber\\
&\leq \sum\limits_{l=1}^s\|y_0^l\|_{\Om}^2=\|y_0\|_{\Om^s}^2.
\end{align}
Combining \eqref{eq_anfangswert} with both statements of the proposition we obtain the boundedness result
\begin{align}\label{eq:totalebeschraenktheit}
\sup_{t\in[0,T]}\|y_n(t)\|_{\Om^s}+\|y_n\|_{W(0,T)^s}\leq C_5\left[\|y_0\|_{\Om^s}+\|f\|_{L^2(0,T;H^1(\Om)^*)^s}\right]
\end{align}
with a constant $C_5>0$ independent of the sequence $(y_n)_{n\in\mathbb{N}}$. In addition, we conclude that all possible solutions of \eqref{ode} are bounded in $C([0,T],\mathbb{R}^{n\times s})$ by a constant only depending on the data of the model. 

Since $F_1$ is Lipschitz continuous, $F_2$ is bounded, $f$ is integrable and Lemma~\ref{hi:bilevans}(1) holds for the bilinear form $B$ problem \eqref{ode}, defined on a bounded rectangular domain, has an absolutely continuous solution $u^n$ in a neighborhood of the initial value by Carath\' eodory's theorem. The a priori estimate allows to choose a domain of definition such that the solution $u^n\in H^1(0,T)^{n\times s}$ is globally defined on $[0,T]$.

For every $n\in\mathbb{N}$, let $y_n\in C([0,T];L^2(\Om))^s$ be defined by the ansatz with the coefficients $u^n=(u^n_1,\dotsc,u^n_n)^{\top}$ obtained by Carath\' eodory's theorem. Since $u^n$ is a weak solution of \eqref{ode} $y_n$ fulfills \eqref{eq:schwach_pw_basis} almost everywhere in $[0,T]$. 
However, the a priori estimate \eqref{eq:totalebeschraenktheit} remains valid for all members of the sequence $(y_n)_{n\in\mathbb{N}}$. In particular, $(y_n)_{n\in\mathbb{N}}$ proves to be bounded in $W(0,T)^s$ which is a Hilbert space and thus reflexive. Therefore, a subsequence $(y_{n_k})_{k\in\mathbb{N}}$ and a limit $y\in W(0,T)^s$ exist with $y_{n_k}\rightharpoonup y$ in $L^2(0,T;H^1(\Om))^s$ and $y_{n_k}^{\prime}\rightharpoonup y^{\prime}$ in $L^2(0,T;H^1(\Om)^*)^s$ for $k\to\infty$.

In the following, we will show that $y$ solves the weak formulation \eqref{eq:schwach_allg}.
Since, in particular, for every $j\in\{1,\dotsc,s\}$ the sequence of the $j$-th components $(y_{n_k}^j)_{k\in\mathbb{N}}$ converges weakly with respect to the norm of $L^2(0,T;H^1(\Omega))$ we conclude for an arbitrary $v\in L^2(0,T;H^1(\Om))^s$:
\begin{eqnarray*}
\begin{split}
&\int_0^T (\kappa_j(t)\nabla  y^j_{n_k}(t),\nabla v^j(t))_{L^2(\Omega)^n}dt \to \int_0^T(\kappa_j(t)\nabla  y^j(t),\nabla v^j(t))_{L^2(\Omega)^n}dt \quad \text {and}\\
&\int_0^T(\vec v(t)\cdot\nabla y^j_{n_k}(t),v^j(t))_{\Om}dt \to \int_0^T(\vec v(t)\cdot\nabla y^j(t),v^j(t))_{\Om}dt
\end{split}
\end{eqnarray*}
and thus $\int_0^T B(y_{n_k},v;t)dt\to\int_0^T B(y,v;t)dt$ if $k\to\infty$. The weak convergence $y_{n_k}^{\prime}\rightharpoonup y^{\prime}$ implies 
 \begin{displaymath}
\int_0^T\langle  y_{n_k}^{\prime}(t),v(t)\rangle dt 
\to
\int_0^T\langle y^{\prime}(t),v(t)\rangle dt
\quad\mbox{for all }v\in L^2(0,T;H^1(\Om))^s.
\end{displaymath}
Analogous results for the operators $F_1$ and $F_2$ depend on the properties of $Y$. 

Let us first consider $W(0,T)^s$ to be compactly embedded in $Y$. Then the bounded sequence $(y_{n_k})_k$ has a subsequence, denoted again by $(y_{n_k})_k$, converging strongly in $Y$. Since strong convergence implies weak convergence and the weak limit is unique we have $y_{n_k}\to y$ in $Y$. 
The continuity of $F:=F_1+F_2$  
yields
 \begin{eqnarray}\label{eq:konvergenznonlin}
\int_0^T\langle  F(y_{n_k}(t)),v(t)\rangle dt 
\to\!
\int_0^T\langle F(y(t)),v(t)\rangle dt
\quad\mbox{for all }v\in L^2(0,T;H^1(\Om))^s.
\end{eqnarray}
To extend the space of admissible test functions for the weak formulation~\eqref{eq:schwach_pw_basis} we choose $m\in \mathbb{N}$ and arbitrary smooth functions $d_{jl}:\left[0,T\right]\to\mathbb{R}$ for all $j=1,\dots,m$, $l=1,\dots,s$. We multiply Eq.~\eqref{eq:schwach_pw_basis} by the proper coefficient $d_{jl}(t)$, summarize over $j=1,\dots,m$ and $l=1,\dots,s$ and integrate with respect to $t$. Since inner products and dual pairings are linear we obtain
\begin{align*}
 \int_0^T\{\langle y^{\prime}_{n_k}(t),v(t)\rangle  + B(y_{n_k},v;t)
 +\langle F(y_{n_k}(t)),v(t)\rangle \} dt= \int_0^T\langle f(t),v(t)\rangle dt
 \end{align*}
with the special test function $v\in C^1([0,T]; H^1(\Om))^s$ defined by the components
\begin{eqnarray}\label{eq:spezielle testfunktion}
v_l=\sum_{j=1}^md_{jl}v_j\in C^1([0,T]; H^1(\Om))\quad\text{for all $l=1,\dotsc,s$.}
\end{eqnarray} 
According to the convergence results above we  obtain by passing to limits 
\begin{align*}
 \int_0^T\{\langle y^{\prime}(t),v(t)\rangle  + B(y,v;t)
 +\langle F(y(t)),v(t)\rangle \} dt= \int_0^T\langle f(t),v(t)\rangle dt
 \end{align*}
for all $v\in C^1([0,T]; H^1(\Om))^s$ of the form \eqref{eq:spezielle testfunktion}. Functions of this type lie dense in $L^2(0,T;H^1(\Om))^s$ since $(v_j)_{j\in\mathbb{N}}$ is a basis of $H^1(\Om)$ and $C^1$ lies dense in $L^2$ \cite[Thm. 3.14]{rud87}. Thus, the weak formulation holds for an arbitrary test function from $L^2(0,T;H^1(\Om))^s$ and, in the first case, the proof is complete.

Now consider the case that $W(0,T)^s\hookrightarrow Y$ is not compactly embedded. 

In the purely Lipschitz continuous case $F_2=0$ we additionally assume that $F_1:Y\to L^2(0,T;H^1(\Om)^*)^s
$ is weakly continuous. Since the weak convergence of $(y_{n_k})_k$ in $W(0,T)^s$ implies the same property in $Y$ we come by the weak convergence of $(F_1(y_{n_k}))_k$ in  $L^2(0,T;H^1(\Om)^*)^s$. Thus, we have a result analogous to \eqref{eq:konvergenznonlin} and the proof is complete.

Consider at last $F_2\neq0$, $L^2(0,T;H^1(\Om))^s\hookrightarrow Y$ and $F_1$ to be strongly continuous. The weak convergence of $(F_1(y_{n_k}))_k$ is deduced as in the last paragraph. Since it is not weakly continuous the same result for $F_2:Y\to L^2(0,T;H^1(\Om)^*)^s$ has to be derived differently.

We observe that the assumptions for $F_2$ and the boundedness of $(y_{n_k})_k$ imply the boundedness of $(F_2(y_{n_k}))_k$ in the Hilbert space $L^2(0,T;H^1(\Omega)^*)^s$. Thus, a subsequence, again denoted by $(F_2(y_{n_k}))_k$, and a limit $D\in L^2(0,T;H^1(\Omega)^*)^s$ exist with $F_2(y_{n_k})\rightharpoonup D$ in $L^2(0,T;H^1(\Omega)^*)^s$. 
 
Therefore, we obtain by passing to limits as in the last paragraph 
\begin{align}\label{eq:weakform_with_D}
\!\!\int_0^T\!\!\!\{\langle y^{\prime}(t),v(t)\rangle+ B(y,v;t)
+ \langle D(t),v(t)\rangle \} dt = \int_0^T\!\!\!\langle f(t)-F_1(y(t)),v(t)\rangle dt
 \end{align}
for all $v\in L^2(0,T;H^1(\Om))^s$.
Strictly speaking, this statement was proved again only for special test functions taken from the space $C^1([0,T]; H^1(\Om))^s$. However, we have already seen above that such functions lie dense in $ L^2(0,T;H^1(\Om))^s$.

%
%
%F(y)=D
%
Before the proof is finished it remains to show $F_2(y)=D$. Since $Y$ is a superset of $L^2(0,T;H^1(\Om))^s$ we are able to employ a lemma from the theory of monotone operators  proved by Gajewski et al. \cite{gaj67} and applied by Tr\"oltzsch \cite{tr02}. Because of the general space $Y$ and the non-monotone operator $F_1$ Tr\"oltzsch's considerations have to be extended. 

We will utilize the statements of the following lemma. As a corollary, we obtain the initial value condition $y(0)=y_0$. 
\begin{lemma}\label{lem:punktweise}
Let $(y_{n_k})_k$ be the sequence defined in the ongoing proof.
Then $(y_{n_k}(t))_k$ converges weakly to $y(t)$ in the space $L^2(\Om)^s$ for every $t\in \left[0,T\right]$. Moreover, the sequence $(y_{n_k}(0))_k$ converges strongly to the initial value $y_0\in L^2(\Om)^s$. In particular, the initial value condition $y(0)=y_0$ is satisfied.
\end{lemma}
\begin{proof}
An easy argument provides that every continuous, linear operator is weakly sequentially continuous, i.e. the image of a weakly convergent sequence is again weakly convergent. 
 
 For every $t\in\left[0,T\right]$ the operator $E_t:C([0,T];L^2(\Om))^s\to L^2(\Om)^s, y\mapsto y(t)$ is obviously linear and bounded due to $\|E_ty\|_{\Om^s}=\|y(t)\|_{\Om^s}\leq\sup_{t\in\left[0,T\right]}\|y(t)\|_{\Om^s}=\|y\|_{C([0,T];L^2(\Om))^s}$. Therefore, it is continuous and thus weakly sequentially continuous. 
  Furthermore, the ongoing proof provides $y_{n_k}\rightharpoonup y$ in the space $C([0,T];L^2(\Om))^s$ because of the embedding $W(0,T)\hookrightarrow C([0,T];L^2(\Om))$. Thus, the weak sequential continuity of $E_t$ implies the first statement of the lemma.

To prove the second assertion we consider the ansatz for $y_{ln_k}(0)$ and the Fourier representation $y_0^l=\sum_{i=1}^\infty (y_0^l,v_i)_{\Om}v_i$ of $y_0^l$ in $L^2(\Om)$ for every $l=1,\dotsc,s$. Estimating their difference we use the properties of inner products and orthonormal bases as in \eqref{eq_anfangswert}. The convergence in the last step results from the quadratic summability of the Fourier coefficients. We obtain
\begin{align*}
\|y_{n_k}(0)-y_0\|^2_{\Om^s}&=\sum_{l=1}^{s}
\|\sum_{i=1}^{n_k}{}^lu_i^{n_k}(0)v_i-\sum_{i=1}^\infty(y_0^l,v_i)_{\Om}v_i\|^2_{\Om}
=\sum_{l=1}^{s}\|\!\!\!\sum_{i=n_k+1}^\infty(y_0^l,v_i)_{\Om}v_i\|^2_{\Om}\\
&=\sum_{l=1}^{s}\sum_{i=n_k+1}^\infty(y_0^l,v_i)_{\Om}^2\to 0\quad\text{for $k\to\infty$.}
\end{align*}
Finally, $y_{n_k}(0)\to y_0$ in $L^2(\Om)^s$ implies the weak convergence $y_{n_k}(0)\rightharpoonup y_0$. On the other hand, the first part of the lemma indicates $y_{n_k}(0)\rightharpoonup y(0)$. The uniqueness of the weak limit yields $y(0)=y_0$.
 \end{proof}
%
%End Lemma 
%
Now we are able to prove the identity $D=F_2(y)$ in the space $L^2(0,T,H^1(\Om)^*)^s$. As announced above, we use the following lemma \cite[Lemma III.1.3]{gaj67}. 
\begin{lemma}
 \label{lem:gajewski} Consider a reflexive Banach space $H$. Let the operator $A:H\to H^*$ be monotone and demi-continuous\footnote{An operator is called demi-continuous if the image of a strongly convergent sequence is weakly convergent. Obviously, continuity implies demi-continuity.}.  If there are $y,y_n\in H$ for all $n\in\mathbb{N}$ and $w\in H^*$ with the properties $y_n\rightharpoonup y$ as well as 
 $$(i)\quad A(y_n)\rightharpoonup w\quad\text{in $H^*$\quad and\quad}(ii)\quad\limsup\limits_{n\to\infty}\langle A(y_n),y_n\rangle_{H^*}\leq\langle w,y\rangle_{H^*}$$
 then $A(y)= w$ in $H^*$.
\end{lemma}
To be conform with the notation of Lemma~\ref{lem:gajewski} we define $H:=L^2(0,T;H^1(\Om))^s$ and restrict the reaction terms to $H$ which is possible because of the assumption $H\hookrightarrow Y$. In addition, we shorten the weak formulation \eqref{eq:weakform_with_D} by
\begin{displaymath}
y^{\prime} + w =R(y),
\end{displaymath}
where the functional $w\in H^*$ and the operator $R:H\to H^*$ are defined by
 \begin{align*}
& \langle w,v\rangle_{H^*}:= \int_0^T\{ B(y,v;t)+\langle D(t),v(t)\rangle \} dt\quad\text{and}\quad\\
& \langle R(\tilde{y}),v\rangle_{H^*}:=\int_0^T\{\langle f(t),v(t)\rangle-\langle F_{1}(\tilde{y}(t)),v(t)\rangle \}dt
\end{align*}
for all $v,\tilde{y}\in H$. Moreover, we define the operator $A:H\to H^*$ by 
\begin{equation*}
 \langle A(\tilde{y}),v\rangle_{H^*}:= \int_0^T\{ B(\tilde{y},v;t)+\langle F_2(\tilde{y}(t)),v(t)\rangle\} dt\quad\text{for all $\tilde{y},v\in H$.}
\end{equation*}

To be able to apply the lemma to $A$ and $w$ we check the assumptions. First, $F_2$ is assumed to be monotone and continuous. The correspondent properties for $B$ are established in Lemma \ref{hi:bilevans}: the monotonicity is stated in \ref{hi:bilevans}(3) and the continuity is equivalent to the boundedness in \ref{hi:bilevans}(1) since $B$ is bilinear. As a consequence, the sum $A$ is also monotone and continuous. Finally, continuity implies demi-continuity.

We have already proved $A(y_{n_k})\rightharpoonup w$ in $H^*$. In order to verify property (ii) of the lemma we deduce from the weak formulation \eqref{eq:schwach_pw_basis}
\begin{equation*}
\int_0^T\langle y_{n_k}^{\prime}(t),y_{n_k}(t)\rangle dt + \langle A(y_{n_k}),y_{n_k}\rangle_{H^*} = \langle R(y_{n_k}),y_{n_k}\rangle_{H^*}
\end{equation*}
 using the definitions of the current proof. 
Applying Theorem \ref{satz:hauptsatz} to the integral on the left side we obtain rearranging the summands 
\begin{eqnarray}
\label{eq:zwischenschrittabs2}
\langle A(y_{n_k}),y_{n_k}\rangle_{H^*} = \langle R(y_{n_k}),y_{n_k}\rangle_{H^*} + \frac{1}{2}\|y_{n_k}(0)\|^2_{\Om^s}-\frac{1}{2}\|y_{n_k}(T)\|^2_{\Om^s}.
 \end{eqnarray}
Lemma \ref{lem:punktweise}, applied to $t=T$, guarantees the weak convergence of $(y_{n_k}(T))_k$ to $y(T)$ which implies $\|y(T)\|_{\Om^s}\leq\liminf_{n\to\infty}\|y_{n_k}(T)\|_{\Om^s}$. Since the upper limit of a real sequence is always greater or equal to the lower limit we deduce
\begin{displaymath}
-\limsup\limits_{k\to\infty}\|y_{n_k}(T)\|^2_{\Om^s}\leq -\liminf\limits_{k\to\infty}\|y_{n_k}(T)\|_{\Om^s}^2\leq -\|y(T)\|_{\Om^s}^2.
\end{displaymath}
The same lemma indicates also $\lim_{k\to\infty}\|y_{n_k}(0)\|^2_{\Om^s}=\|y(0)\|^2_{\Om^s}$.

Now we investigate the convergence of $\langle R(y_{n_k}),y_{n_k}\rangle_{H^*}$. Since $f$ belongs to $H^*$ the weak convergence of $(y_{n_k})_k$  in $H$ provides $\langle f,y_{n_k}\rangle_{H^*}\to\langle f,y\rangle_{H^*}$ for $k\to\infty$.

On the other hand, we conclude by the strong continuity of $F_1$ 
\begin{align*}
|\langle F_{1}(y_{n_k}),y_{n_k}\rangle _{H^*}-\langle F_{1}(y),y\rangle _{H^*}|&\leq
|\langle F_{1}(y_{n_k})-F_1(y),y_{n_k}\rangle _{H^*}|+|\langle F_{1}(y),y_{n_k}-y\rangle _{H^*}| \\
&\leq
\| F_{1}(y_{n_k})-F_1(y)\|_{H^*}\|y_{n_k}\| _{H}+|\langle F_{1}(y),y_{n_k}-y\rangle _{H^*}|.
\end{align*}
The first summand converges to zero because $F_{1}(y_{n_k})\to F_1(y)$ in $H^*$ while the weak convergence of $(y_{n_k})_k$ induces its boundedness in $H$. For the same reason also the second summand converges to zero since $F_{1}(y)\in H^*.$ Altogether, the convergence $\langle R(y_{n_k}),y_{n_k}\rangle_{H^*}\to \langle R(y),y\rangle_{H^*}$ holds.

By these results we obtain for the upper limit of Eq.~(\ref{eq:zwischenschrittabs2}):
\begin{equation*}
\begin{split}
 \limsup\limits_{n\to\infty}\langle A(y_{n_k}),&y_{n_k}\rangle_{H^*}\!=\! \lim\limits_{n\to\infty}\!( \langle R(y_{n_k}),y_{n_k}\rangle_{H^*}\!+\! \frac{1}{2}
 \|y_{n_k}\!(0)\|^2_{\Om^s})\!-\!\frac{1}{2}\!\limsup\limits_{n\to\infty}\!\|y_{n_k}\!(T)\|_{\Om^s}^2\\
&\leq \langle R(y),y\rangle_{H^*}+ \frac{1}{2}\|y(0)\|^2_{\Om^s}-\frac{1}{2}\|y(T)\|^2_{\Om^s}\\
& = \langle R(y),y\rangle_{H^*}\!-\! \int_0^T\!\!\langle y^{\prime}(t),y(t)\rangle dt
= \langle R(y),y\rangle_{H^*}-\langle y^{\prime},y\rangle_{H^*} = \langle w,y\rangle_{H^*}.
 \end{split}
\end{equation*} 
In the second line, Theorem \ref{satz:hauptsatz} is applied again. The obtained integral is perceived as an element of $H^*$. The last equality sign is valid because $y\in H$ both fulfills the weak formulation and defines a proper test function.

Thus, Lemma~\ref{lem:gajewski} yields $A(y)=w$, i.e.
\begin{displaymath}
 \int_0^T\{B(y,v;t)+\langle D(t),v(t)\rangle\}dt= \int_0^T\{B(y,v;t)+\langle F_2(y(t)),v(t)\rangle\}dt
\end{displaymath}
for all $v\in H$. By subtracting the bilinear summand on both sides we obtain $D=F_2(y)$ in $H^*$.\\
%End proof troeltzsch

%Uniqueness:
Having shown the existence of a solution the proof of uniqueness remains. Let therefore $y_1,y_2\in W(0,T)^s$ be two weak solutions of the initial value problem \eqref{eq:schwach_allg}. It has to be shown that the difference $y:=y_1-y_2$ equals zero. Since both $y_1$ and $y_2$ have the same initial value we conclude $y(0)=y_1(0)-y_2(0)=y_0-y_0=0$. Inserting an arbitrary test function $v\in H^1(\Om)^s$ into the weak formulations for $y_1(t)$ and $y_2(t)$ and subtracting these equations we deduce
\begin{equation*}
 \langle y^{\prime}(t),v\rangle+B(y,v;t) + \langle F_{1}(y_{1}(t))-F_{1}(y_{2}(t)),v\rangle +\langle F_{2}(y_{1}(t))-F_{2}(y_{2}(t)),v\rangle =0
 \end{equation*}
for almost all $t\in\left[0,T\right]$. Thus, we are in the situation of Proposition~\ref{prop:energy_estimates} with $Z=H^1(\Om)$, $f=0$, $z_i=y_i\in W(0,T)^s$. We obtain in particular
\begin{eqnarray*}
\|y_1-y_2\|_{L^2(0,T;H^1(\Om))^s}\leq C(\|y_1(0)-y_2(0)\|_{\Om^s}+\|0\|_{L^2(0,T;H^1(\Om)^*)^s})=0.
\end{eqnarray*}
Since the norm is positive definite the result $y_1-y_2=0$ follows immediately.
Thus, the proof is complete.
\end{proof}
%%%%%%END proof GALERKIN
%%%%
%%%%
\section{An existence and uniqueness result with Banach's Fixed Point Theorem}\label{sec:banach}
This section contains a second existence and uniqueness result. In the proof, the solution is identified with a fixed point of a certain map and determined by means of Banach's Fixed Point Theorem: 
\begin{theorem}\label{sa:banachscherfixpunkt}\emph{\textbf{(Banach)}} Let $X$ be a Banach space and the map $A:X\to X$ is Lipschitz continuous with a constant $L\in (0,1)$. Hence a unique fixed point of $A$ exists in $X$, i.e. there is $x^{*}\in X$ with the property $A(x^{*})=x^{*}$.
\end{theorem}
The proof of Banach's Fixed Point Theorem, carried out e.g. by Zeidler \cite{zei85}, is constructive: the fixed point is identified with the limit of a specific sequence. Thus, the proof of the following existence theorem, based on Thm.~\ref{sa:banachscherfixpunkt}, provides an algorithm that helps compute the weak solution numerically.% This result is of special interest if the solution actually has to be determined.
%

 %THEOREM BANACH
\begin{theorem}\label{satz:ex_speziell}
Let $Y\hookrightarrow L^2(0,T;L^2(\Om))^s$ be a Banach space with $C([0,T];L^2(\Om))^s%\cap L^2(0,T;H^1(\Om))^s
\hookrightarrow Y$. As in Thm.~\ref{satz:ex_allgemein}, consider $F_1$ to be Lipschitz continuous, $F_2$ to be monotone and bounded and the condition $F_2(0)=0$ to be fulfilled.
In case $F_2\neq 0$, either the embedding $W(0,T)^s\hookrightarrow Y$ is compact or $L^2(0,T;H^1(\Om))^s\hookrightarrow Y$. Then, there is a unique solution $y\in W(0,T)^s$ of the weak initial value problem in Eq.~\eqref{eq:schwach_allg}. Furthermore, the estimate 
\begin{displaymath}
 \|y\|_{W(0,T)^s}\leq C(%\|d(0,.\,,.\,)\|_{Q_T^s}+\|b(0,.\,,.\,)\|_{\Sigma^s}
 \|f\|_{L^2(0,T;H^1(\Om)^*)^s}
 +\|y_0\|_{\Omega^s})
\end{displaymath}
holds with a constant $C>0$ independent of $y$ and $y_0$.
\end{theorem}
%
%
%
%
%
%
%%%%%%%%%%PROOF WITH BANACH
%
%\subsection*{Proof of Theorem}
\begin{proof}[Proof of Thm.~\ref{satz:ex_speziell}]
We will extend the method used by Evans \cite[Section 9.2]{evans} for one-dimensional, purely Lipschitz continuous problems and homogeneous boundary conditions. Banach's Fixed Point Theorem will be applied to the space $X:=C([0,T];L^2(\Om))^s$, endowed with the norm $\|y\|_{C}^2:=\sup_{t\in[0,T]}\|y(t)\|_{L^2(\Om)^s}^2\text{e}^{-Ct}$. The constant $C>0$ is a priori arbitrary and will be specified later on. As this modified norm is equivalent to the usual maximum norm $(X,\|.\|_C)$ defines a Banach space.

The following proof bases on the idea of approximating a solution of problem \eqref{eq:schwach_allg} by solutions of purely monotone problems. The Lipschitz continuous reaction term is eliminated by inserting a fixed $z\in X$. 
Since $F_1(z)\in L^2(0,T;H^1(\Om)^*)^s$ % into account Lemma~\ref{lem:db_erz_funkt} 
the monotone, inhomogeneous problem 
\begin{align}\label{eq:schwach_linearisiert}
 y^{\prime} +\int_0^T B(y,.\,;t)dt 
 +F_2(y) &= f-F_1(z)\\
y(0)&=y_0\nonumber
 \end{align}
 is well-defined
 and has a unique weak solution $y(z)\in W(0,T)^s$ because of Theorem~\ref{satz:ex_allgemein}. Due to Thm.~\ref{satz:hauptsatz}(1)
\begin{displaymath}
A:X\to X,\,\,z\mapsto y(z)
\end{displaymath}
maps $X$ into itself. Obviously, $y$ is a fixed point of $A$ if and only if it solves the original problem~\eqref{eq:schwach_allg}.

Thanks to Banach's Fixed Point Theorem it suffices to show the Lipschitz continuity of $A$ with a constant in the interval $(0,1)$. Choose therefore $z_1,z_2\in X$ and abbreviate $y_i:=A(z_i)$ for $i\in\{1,2\}$.

To establish an estimate for the difference $\delta:=y_1-y_2$ we consider the weak formulations for $y_i(t)$ for almost every $t\in[0,T]$. Being elements of $(H^1(\Om)^*)^s$, their summands  can be applied to any $v \in H^1(\Om)^s$. Subtracting the equations from each other we obtain due to the linearity of the first two summands on the left side 
\begin{align}\label{eq:absch_nichtlin_beide}
\langle\delta^{\prime}(t),v\rangle+B(\delta,v;t)+\langle F_2(y_1(t))-F_2(y_2(t)),v\rangle
&=\langle F_1(z_2(t))-F_1(z_1(t)),v\rangle.
 \end{align}
The inhomogeneity $f$ vanishes since it appears in both of the weak formulations.
This equation corresponds to \eqref{eq:prop_eq}. The assumptions
allow to apply the first part of Prop.~\ref{prop:energy_estimates} yielding, in particular, the analog to Eq.~\eqref{eq:zitatbewbanach}
\begin{equation*}
\begin{split} 
\|\delta(t)\|_{\Om^s}^2 \leq \text{e}^{tc_1}\int_0^t  \frac{1}{2\varepsilon}\|F_1(z_2(\sigma))-F_1(z_1(\sigma))\|^2_{(H^1(\Om)^*)^s} d\sigma
\end{split}
\end{equation*}
for all $t\in[0,T]$ with the constant $c_1=2\kappa_{\min}>0$. Since $A(z_1)(0)=y_0=A(z_2)(0)$ the initial value of $\delta$ is equal to zero.
Applying the assumed Lipschitz condition of $F_1$ we arrive at 
\begin{equation*}
\begin{split} 
\|\delta(t)\|_{\Omega^s}^2& \leq\text{e}^{tc_1}\int_{0}^t \! \! \Psi\|(z_1-z_2)(\sigma)\|_{\Omega^s}^2d\sigma
\end{split}
\end{equation*}
 with the constant $\Psi:=L_1^2/(2\varepsilon)$.
In the next step, we estimate the exponential function and extend the integrand with respect to the underlying maximum norm. We obtain
\begin{align*}%\label{eq:bewendesup}
\|\delta(t)\|_{\Omega^s}^2 &\leq \text{e}^{Tc_1}\int_{0}^t \! \! \Psi\|(z_1-z_2)(\sigma) \|_{\Omega^s}^2\text{e}^{-C\sigma}\text{e}^{C\sigma}d\sigma
\leq \Psi\text{e}^{Tc_1}\|z_1-z_2\|_{C}^2\int_{0}^t \! \! \text{e}^{C\sigma}d\sigma\\
&\leq
\frac{\Psi\text{e}^{Tc_1}}{C}\|z_1-z_2\|_{C}^2 \text{e}^{Ct}.
\end{align*}
In the last step, the remaining integral was estimated by
\begin{eqnarray}
\int_0^t\mbox{e}^{C\sigma}d\sigma=\frac{1}{C}[\text{e}^{Ct}-1]\leq\frac{1}{C}\text{e}^{Ct}.\label{eq:integralberechnung}
\end{eqnarray}
After having multiplied both sides of the inequality for $\|\delta(t)\|_{\Omega^s}^2$ by $\text{e}^{-Ct}$ we find out for the supremum 
\begin{displaymath}
\|A(z_1)-A(z_2)\|_{C}^2=\sup_{t\in[0,T]}\|\delta(t)\|_{\Om^s}^2 \text{e}^{-Ct}\leq \frac{\Psi\text{e}^{Tc_1}}{C}\|z_1-z_2\|_{C}^2.
\end{displaymath}
Thus, $A$ proves to be Lipschitz continuous with the constant 
\begin{displaymath}
L_A:=\sqrt{\frac{1}{C}\frac{L_1^2}{2\varepsilon}\text{e}^{2T\kappa_{\min}}}.
\end{displaymath}
The proof is valid for any $C>0$.  Choosing $C>L_1^2(2\varepsilon)^{-1}\exp(2T\kappa_{\min})$ we obtain the property $L_A<1$ due to the strict monotonicity of the square root function on $\mathbb{R}_{>0}$. Hence the map $A$ is a contraction in the Banach space $X$, endowed with the modified maximum norm with the special $C$. Banach's theorem provides the existence of a unique fixed point $y\in X$ of $A$.
Since every element in $X$ is a fixed point if and only if it solves Eq.~\eqref{eq:schwach_allg} the proof of existence and uniqueness is complete.

The asserted estimate of the solution $y$ is a direct consequence of Proposition~\ref{prop:energy_estimates}. By inserting an arbitrary element $v\in H^1(\Om)^s$ as a test function into the weak formulation for $y(t)$ we obtain 
\begin{align*}
\langle y^{\prime}(t), v \rangle +B(y,v;t)+\langle F_1(y(t)),v\rangle+\langle F_2(y(t)),v\rangle 
= \langle f(t),v\rangle
\end{align*}
for almost every $t\in[0,T]$ which corresponds to Eq.~\eqref{eq:prop_eq} with $z_1=z=y$ and $z_2=0$. Prop.~\ref{prop:energy_estimates} yields a constant $C_W>0$ with
\begin{eqnarray*} \|y\|_{W(0,T)^s}
  \leq C_W(\|f\|_{L^2(0,T;H^1(\Om)^*)^s}+\|y(0)\|_{\Om^s}).
\end{eqnarray*}
Taking into account the initial value condition $y(0)=y_0$, the proof is complete.
\end{proof}
%
%
%%APPLICATIONS
%
\section{Analysis of the $PO_4$-$DOP$-model}\label{sec:examples}
In this section we will apply the results about existence and uniqueness to the initial value problems associated with the $PO_4$-$DOP$-model. In Sec.~\ref{sec:modellgleichungenndop} we introduced them as the model equations themselves (Eq.~\eqref{eq:zustand}) and their derivative (Eq.~\eqref{eq:zustand_derivative}).

In the $PO_4$-$DOP$-model, the biological uptake of phosphate is expressed by means of saturation functions. Since this kind of function is very typical for marine ecosystem models we will investigate it on a more abstract level in the next subsection.  

\subsection{Saturation functions}\label{sub:saturation}
Reactions in marine ecosystems, e.g. the growth of a tracer or the transformation of one tracer into another, are often described by Michaelis-Menten kinetics. According to this theory, the reaction rate does not increase proportionally with the influencing factors (e.g. nutrients or light) but approaches a maximum rate at high concentrations of the influencing factors. This is typically expressed by a saturation function like
\begin{displaymath}
f_K:\mathbb{R}\to\mathbb{R}, \quad f_K(x):=\frac{x}{|x|+K}.
\end{displaymath}
The half saturation constant $K>0$ indicates the concentration at which the reaction rate is half of the maximum.

Variants of the function $f_K$ are found in many ecosystem models. Examples are the $PO_4$-$DOP$-model or the $NPZD$-model of Schartau and Oschlies, presented  by R\"uckelt et al. \cite{rue10}. In general, the modulus in the denominator does not appear in the actual model descriptions since, naturally, tracer concentrations are supposed to be positive. However, it cannot be omitted in a strict mathematical formulation since a priori it is not known whether the solution of a partial differential equation is nonnegative. 

In the next lemma, we state some essential properties of $f_K$.
\begin{lemma}\label{hi:hilfsfktreel}
 The real function $f_K$ is bounded by 1 and Lipschitz continuous.
\end{lemma}
\begin{proof}
If $x\neq0$ we have $|x|+K\geq|x|$ and therefore
 \begin{displaymath}
|f_K(x)|=|\frac{x}{|x|+K}|\leq\frac{|x|}{|x|}=1.
\end{displaymath}
Since the same statement obviously holds for $x=0$ the function $f_K$ is bounded by 1. The Lipschitz continuity is proved by virtue of the well-known mean value theorem. Therefore, we show that $f_K$ is differentiable. Due to the modulus $|x|$ in the denominator the differentiability in $x=0$ has to be regarded separately. However, since the limits exist we conclude:
\begin{equation*}
\begin{split}
&f_K^{\prime}(0+)=\lim_{t\downarrow0}\frac{f_K(0+t)-f_K(0)}{t}=\lim_{t\downarrow0}\frac{1}{t}\frac{t}{t+K}=\lim_{t\downarrow0}\frac{1}{t+K}=\frac{1}{K}=\frac{K}{(|0|+K)^2}\mbox{,}\\
&f_K^{\prime}(0-)=\lim_{t\uparrow0}\frac{f_K(0+t)-f_K(0)}{t}=\lim_{t\uparrow0}\frac{1}{t}\frac{t}{-t+K}=\lim_{t\uparrow0}\frac{1}{-t+K}=\frac{1}{K}=\frac{K}{(|0|+K)^2}.
\end{split}
 \end{equation*}
Both of the one-sided limits are equal and thus $f_K$ is differentiable in $x=0$. Everywhere else the differentiability follows from the fact that $f_K$ is a composition of differentiable functions. The derivative can be determined via the quotient rule:

\begin{align*}
&f_K^{\prime}(x)=\frac{d}{dx}\frac{x}{x+K}=\frac{x+K-x}{(x+K)^2}=\frac{K}{(x+K)^2}=\frac{K}{(|x|+K)^2}&&\mbox{ for $x>0$,}\\
&f_K^{\prime}(x)=\frac{d}{dx}\frac{x}{-x+K}=\frac{-x+K+x}{(-x+K)^2}=\frac{K}{(-x+K)^2}=\frac{K}{(|x|+K)^2}&&\mbox{ for }x<0.
 \end{align*}
 From $|x|+K\geq K$ we conclude $|f^{\prime}_K(x)|=K/(|x|+K)^2\leq1/K$. The mean value theorem yields
\begin{displaymath}
|f_K(x)-f_K(y)|\leq\max_{\xi\in\mathbb{R}}|f^{\prime}_K(\xi)||x-y|=\frac{1}{K}|x-y|\quad\mbox{for all } x,y\in\mathbb{R},
\end{displaymath}
the Lipschitz continuity of $f_K$ with the constant $1/K$. 
\end{proof}

%NDOP Modell Untersuchung
%
\subsection{The $PO_{4}$-$DOP$-model equations}\label{sub:ndop}
In this section, we show the unique weak solvability of the $PO_4$-$DOP$-model equations \eqref{eq:zustand} by virtue of Theorem~\ref{satz:ex_speziell}. 

Both reaction terms $d$ and $b$  fulfill the assumptions of Sec.~\ref{sec:generalassumptions} concerning the generating functionals and their domain of definition $Y:=L^2(0,T;L^2(\Om))^2$ has the property $C([0,T];L^2(\Om))^2\hookrightarrow Y$. The operators $F_1:Y\to L^2(0,T;H^1(\Om)^*)^2,\,y\mapsto \tilde{d}(y)+\tilde{b}(y)$, defined according to Lem.~\ref{lem:db_erz_funkt}, and $F_2=0$ represent the reaction terms belonging to the weak formulation of the $PO_{4}$-$DOP$-model. Furthermore, the proof of Lem.~\ref{lem:db_erz_funkt} and the triangle inequality yield  
\begin{align*}
\|F_1(y(t))-F_1(z(t))\|_{(H^1(\Om)^*)^2}\!\leq\! \|d(y,.\,,t)-d(z,.\,,t)\|_{\Om^2}\!+\!c_{\tau}\|b(y,.\,,t)-b(z,.\,,t)\|_{\Ga^2}
\end{align*}
for all $y,z\in Y$. Thus, it suffices to prove the Lipschitz continuity of the functions $d$ and $b$.

As a preparation, we establish this property for the components $G,E$ and $\bar{F}$. To this end, choose $t\in[0,T]$ and $y,z\in L^2(0,T;L^2(\Om))^2$. Employing notation and results of Lemma~\ref{hi:hilfsfktreel} we obtain primarily
\begin{equation*}
\begin{split}
\|G(y_1,.\,,t)\!-G(z_1,.\,,t)\|^2_{\Om_1}&\!=\!\int_{\Om_1}\!\!\!\!\alpha^2
f_{K_I}^2(I(x^{\prime\!\!\!},t)\text{e}^{-x_3K_W})
|f_{K_P}(y_1(x,t))\!-f_{K_P}(z_1(x,t))|^2dx\\
&\!\leq\!\int_{\Om_1}\!\!\!\!\alpha^2\frac{1}{K^2_P}|y_1(x,t)-z_1(x,t)|^2dx=\frac{\alpha^2}{K^2_P}\|y_1(t)-z_1(t)\|^2_{\Om_1}.
\end{split}
\end{equation*} 
Considering $E$, we apply H\"older's inequality to the integral over $[0,h_e(x^{\prime})]$. Since $h_e(x^{\prime})\leq \bar{h}_e$ and the latter is independent of $x^{\prime}$ we arrive at an integral over $\Om_1$. In the last line we insert the result obtained for $G$. Thus, we obtain
\begin{equation*}
\begin{split}
\|E(y_1,.\,,t)&-E(z_1,.\,,t)\|^2_{\Om^{\prime}}
=\int_{\Om^{\prime}}
(1-\nu)^2(\int_0^{h_e(x^{\prime})}\{G(y_1,x,t)-G(z_1,x,t)\}dx_3)^2
dx^{\prime}\\
&\leq
(1-\nu)^2\int_{\Om^{\prime}}
h_e(x^{\prime})\int_0^{h_e(x^{\prime})}\{G(y_1,x,t)-G(z_1,x,t)\}^2dx_3
dx^{\prime}\\
&\leq(1-\nu)^2\bar{h}_e\|G(y_1,.\,,t)-G(z_1,.\,,t)\!\|^2_{\Om_1}\!\leq\!
\frac{\alpha^2(1-\nu)^2\bar{h}_e}{K^2_P}\|y_1(t)-z_1(t)\|^2_{\Om}.
\end{split}
\end{equation*}
This computation shows clearly how the norm in the two-dimensional space $\Om^{\prime}$ is transformed into a norm in the three-dimensional space $\Om_1$ by the non-locality of $E$. Without this property, the result for $G$ would not have been applicable. 

In order to show the analogous condition for $\bar{F}$ we observe for an arbitrary $\gamma>0$
\begin{eqnarray}\label{eq:hilf_tiefekleiner1}
\left(\frac{x_3}{\bar{h}_e}\right)^{-\gamma}=\left(\frac{\bar{h}_e}{x_3}\right)^{\gamma}\leq\left(\frac{\bar{h}_e}{\bar{h}_e}\right)^{\gamma}=1\quad\mbox{ for all $(x^{\prime\!\!\!},x_3)\in\overline{\Om}_2$}
\end{eqnarray}
since the component indicating depth fulfills $x_3\in [\bar{h}_e,h(x^{\prime})]$ in the aphotic zone.
 
At last, we consider $\bar{F}$. From \eqref{eq:hilf_tiefekleiner1} with $\gamma=2(\beta+1)$ we obtain an estimate of the integrand independent of $x_3$. Thus, the integral over $[\bar{h}_e,h(x^{\prime})]$ vanishes. Considering the finite maximal depth $h_{\max}$ and the inclusion $\Om_2^{\prime}\subseteq\Om^{\prime}$ we are able to employ the Lipschitz property of $E$. These arguments lead to
 \begin{equation*}
\begin{split}
 \|\bar{F}(y_1,.\,,t)-\bar{F}(z_1,.\,,t)\|^2_{\Om_2}
 &\!=
  \!\!\int_{\Om_2^{\prime}}\!\int_{\bar{h}_e}^{h(x^{\prime})}\!\!
 \frac{\beta^2}{\bar{h}_e^2}\!\left(\frac{x_3}{\bar{h}_e}\right)^{\!\!\!-2(\beta+1)} \!\!\!(E(y_1,x^{\prime\!\!\!},t)-E(z_1,x^{\prime\!\!\!},t))^2dx_3dx^{\prime}\\
&\leq\int_{\Om_2^{\prime}}  \frac{\beta^2}{\bar{h}_e^2}(h(x^{\prime})-\bar{h}_e)(E(y_1,x^{\prime\!\!\!},t)-E(z_1,x^{\prime\!\!\!},t))^2dx^{\prime}\\
& \leq\frac{\beta^2}{\bar{h}_e^2}(h_{\max}-\bar{h}_e)\|E(y_1,.\,,t)-E(z_1,.\,,t)\|^2_{\Om_2^{\prime}} \\
&\leq\frac{\beta^2}{\bar{h}_e^2}(h_{\max}-\bar{h}_e)\frac{\alpha^2(1-\nu)^2\bar{h}_e}{K^2_P}\|y_1(t)-z_1(t)\|^2_{\Om}.%\\
%&=\left(\frac{h_{\max}}{\bar{h}_e}-1\right)\frac{\alpha^2(1-\nu)^2\beta^2}{K^2_P}\|y_1(t)-z_1(t)\|^2_{\Om}.
\end{split}
\end {equation*}
The preliminaries lead to the Lipschitz properties of $d$ and $b$. As to $d_1$, we conclude applying the triangle inequality in combination with the convexity of the square function on $\mathbb{R}$
\begin{equation*}
\begin{split}
 \|&d_1(y,.\,,t)-d_1(z,.\,,t)\|^2_{\Om}=\|-\lambda y_2(t)+G(y_1,.\,,t)+\lambda z_2(t)-G(z_1,.\,,t)\|^2_{\Om_1}\\
 &\qquad\qquad\qquad\qquad\qquad\qquad\quad+\|-\lambda y_2(t)+\bar{F}(y_1,.\,,t)+\lambda z_2(x)-\bar{F}(z_1,.\,,t)\|^2_{\Om_2}\\
&\leq
2(\lambda^2 \|y_2(t)- z_2(t)\|^2_{\Om}\!+\|G(y_1,.\,,t)-G(z_1,.\,,t)\|^2_{\Om_1}
 \!+\|\bar{F}(y_1,.\,,t)-\bar{F}(z_1,.\,,t)\|^2_{\Om_2})\\
 &\leq 2( \lambda^2 \|y_2(t)- z_2(t)\|^2_{\Om}+ \frac{\alpha^2}{K_P^2}(1+\left(\frac{h_{\max}}{\bar{h}_e}-1\right)\beta^2(1-\nu)^2)\|y_1(t)-z_1(t)\|^2_{\Om})\\
 &\leq2\max\{ \lambda^2, \frac{\alpha^2}{K^2_P}(1+\left(\frac{h_{\max}}{\bar{h}_e}-1\right)\beta^2(1-\nu)^2) \}\|y(t)-z(t)\|^2_{\Om}.
\end{split}
\end{equation*} 
For $d_2$ we conclude similarly:
\begin{equation*}
\begin{split}
 \|d_2(y,.\,,t)-d_2(z,.\,,t)\|^2_{\Om}&=\|\lambda y_2(t)-\nu G(y_1,.\,,t)- \lambda z_2(t)+\nu G(z_1,.\,,t)\|_{\Om_1}^2\\
 &\qquad\qquad\qquad\qquad\qquad\qquad\qquad\!\!+\lambda^2\| y_2(t)- z_2(t)\|^2_{\Om_2}\\
 &\leq2(\lambda^2 \|y_2(t)- z_2(t)\|^2_{\Om}+\nu^2\|G(y_1,.\,,t)-G(z_1,.\,,t)\|^2_{\Om_1})\\
& \leq2\max\{ \lambda^2, \frac{\alpha^2\nu^2}{K_P^2} \}\|y(t)-z(t)\|^2_{\Om}.
\end{split}
\end{equation*} 
Thus, the function $d$ fulfills the Lipschitz condition on the product space with a constant given by
 \begin{displaymath}
L_d^2=2(\max\{ \lambda^2, \frac{\alpha^2}{K^2_P}(1+\left(\frac{h_{\max}}{\bar{h}_e}-1\right)\beta^2(1-\nu)^2) \}+\max\{ \lambda^2, \frac{\alpha^2\nu^2}{K_P^2} \}).
\end{displaymath}
The non-zero boundary reaction term $b_1$ is treated in a similar manner. Taking into account $b_1=0$ on $\Ga^{\prime}$ the parametrization of the boundary turns the corresponding norm into the following integrals over the surface $\Om^{\prime}$:
\begin{equation*}
\begin{split}
 \|b_1(y,.\,,t)-b_1(z,.\,,t)&\|^2_{\Gamma}=\int_{\Om_1^{\prime}}(E(y_1,x^{\prime\!\!\!},t)-E(z_1,x^{\prime\!\!\!},t))^2dx^{\prime}\\
&\qquad\qquad\qquad +\int_{\Om_2^{\prime}}(E(y_1,x^{\prime\!\!\!},t)-E(z_1,x^{\prime\!\!\!},t))^2\left(\frac{h(x^{\prime})}{\bar{h}_e}\right)^{\!\!\!-2b}dx^{\prime}\\
 &\leq
 \|E(y_1,.\,,t)-E(z_1,.\,,t)\|^2_{\Om^{\prime}}\leq
 \frac{\alpha^2(1-\nu)^2\bar{h}_e}{K^2_P}\|y_1(t)-z_1(t)\|^2_{\Om}.
\end{split}
\end{equation*} 
The estimation in the last line uses \eqref{eq:hilf_tiefekleiner1} and $\Om_1^{\prime}\cup\Om_2^{\prime}=\Om^{\prime}$ as well as the above result for $E$.
Since $b_2=0$ the Lipschitz constant for $b$ is given by 
 \begin{displaymath}
L_b= \frac{\alpha(1-\nu)\sqrt{\bar{h}_e}}{K_P}.
\end{displaymath}
The actual appearance of the Lipschitz constants is of interest in both the determination of the solution's upper bounds (cf. Prop.~\ref{prop:energy_estimates}) and its computation by means of the algorithm derived from Banach's Fixed Point Theorem (cf. Thm.~\ref{satz:ex_speziell}). In the second case, the Lipschitz constants determine the norm of the solution space.
\subsection{The derivative}\label{sub:derivatives}
In this subsection, we solve the derivative of the $PO_4$-$DOP$-model, given in Eq.~\eqref{eq:zustand_derivative}.
The corresponding weak formulation has the form
\begin{align*}
 \int_0^T\{\langle h^{\prime}(t),w(t)\rangle + B(h,w;t)+(\partial_yd(y)h&(t),w(t))_{\Om^2}+(\partial_yb(y)h(t),w(t))_{\Ga^2}\}dt\\
 & = \int_0^T\{(f(t),w(t))_{\Om^s}+(g(t),w(t))_{\Ga^s}\}dt
 \end{align*}
for all test functions $w\in L^2(0,T;L^2(\Om))^2$. 
Here, $h$ is the unknown and $y$ denotes the solution of the non-linearized equation.
The initial value is zero since it is independent of the parameters.
The inhomogeneities $f\in L^2(0,T;L^2(\Om))^2$ and $g\in L^2(0,T;L^2(\Ga))^2$ represent the derivatives of $d$ and $b$ with respect to the parameters. 

As a first step, we determine the Fr\'echet-derivatives of $d$ and $b$. Both operators are based on the auxiliary operator $G$. By definition, $G$ originates from the real function $f_K$, defined in Sec.~\ref{sub:saturation}, multiplied by an essentially bounded function of space and time.
Operators on function spaces originating from real functions are called superposition or Nemytski operators. The superposition operator $G$ is Fr\'echet-differentiable between the spaces $L^p(Q_T)=L^p(0,T;L^p(\Om))$ with $p>2$ and $L^2(Q_T)$ and the Fr\'echet-derivative is given by the product with the derivative of the underlying real function \cite [Thm.~3.13]{ap90}. Taking into account the proof of Lem.~\ref{hi:hilfsfktreel}, the derivative $\partial_yG(y_1)\in\mathcal{L}(L^p(Q_T),L^2(Q_T))$ is defined by 
\begin{displaymath}
\partial_yG(y_1)h_1 (x,t)=
\alpha\frac{K_Ph_1(x,t)}{(|y_1(x,t)|+K_P)^2}\frac{I(x^{\prime\!\!\!},t)\text{e}^{-x_3K_W}}{|I(x^{\prime\!\!\!},t)\text{e}^{-x_3K_W}|+K_I}\text{ for all }h_1\in L^p(Q_T).
\end{displaymath}
With this preliminary work, $\partial_yE\in\mathcal{L}(L^p(Q_T),L^2(\Om^{\prime}\times[0,T]))$ can be determined as
\begin{displaymath}
\partial_yE(y_1)h_1(x^{\prime\!\!\!},t)=(1-\nu)\int_0^{h_e(x^{\prime})}\alpha\frac{K_Ph_1((x^{\prime\!\!\!},x_3),t)}{(|y_1((x^{\prime\!\!\!},x_3),t)|+K_P)^2}\frac{I(x^{\prime},t)\text{e}^{-x_3K_W}}{|I(x^{\prime\!\!\!},t)\text{e}^{-x_3K_W}|+K_I}dx_3.
\end{displaymath}
This also determines $\partial_y\bar{F}(y_1)$ since $\bar{F}$ is the product of $E$ with a bounded factor independent of $y$. The remaining parts of the reaction terms are linear and bounded and hence also Fr\'echet-differentiable. 

As to the assumptions of Thm.~\ref{satz:ex_allgemein}, the above results suggest to choose the domain of definition $Y:=L^3(Q_T)^2$. The corollaries following Lemma~2.74 of R\r{u}\v{z}i\v{c}ka \cite{ru10} indicate that $W(0,T)^2$ is compactly embedded in $Y$. In particular, it is allowed to insert the non-linearized solution $y\in W(0,T)^2$ into $\partial_yd$ and $\partial_yb$. The ``reaction term'' $F_1:Y\to L^2(0,T;H^1(\Om)^*)^2$ is defined
by 
\begin{displaymath}
 \langle F_1(h),w \rangle_{L^2(0,T;H^1(\Om)^*)}:= \int_0^T\{(\partial_yd(y)h(t),w(t))_{\Om^2}+(\partial_yb(y)h(t),\tau w(t))_{\Ga^2}\}
 dt
 \end{displaymath}
for all $h\in Y$, $w\in L^2(0,T;H^1(\Om))^2$. Obviously, $F_1$ is linear and continuous and therefore weakly continuous (cf. proof of Lemma~\ref{lem:punktweise}). 

Further, the same argumentation as above, applied to $d(t):L^3(\Om)\to L^2(\Om)$,  yields the existence of $\partial_yd(t)(v)\in \mathcal{L}(L^3(\Om)^2, L^2(\Om)^2)$ 
for all $v\in L^3(\Om)^2$. Taking into account the embedding $L^3(\Om)\hookrightarrow H^1(\Om)$ and Lem.~\ref{lem:db_erz_funkt}, $\partial_yd(t)(v)$ can be perceived as an element of  $\mathcal{L}(H^1(\Om)^2, (H^1(\Om)^*)^2)$. Comparing with the results for the spaces involving time we observe $\partial_yd(t)(y(t))[h(t)]=[\partial_yd(y)h](t)$ for every $h\in Y$ and almost all $t\in[0,T]$. 

Since the considerations of the last paragraph analogously hold for $b(t)$ the operators $F_1(t):=\partial_yd(t)(y(t))+ \partial_yb(t)(y(t)):H^1(\Om)^2\to (H^1(\Om)^*)^2$ generate $F_1$ in the sense of Eq.~\eqref{eq:generating}. 

In order to apply Thm.~\ref{satz:ex_allgemein}, $F_1(t)$ is required to be Lipschitz continuous with respect to the norms of $L^2(\Om)^2$ and $(H^1(\Om)^*)^2$. 
As in the last section, it suffices to establish this condition for $\partial_yd(t)(y(t))$ and $\partial_yb(t)(y(t))$ with respect to the appropriate norms. %%

We have seen that, given arbitrary elements $y_1,h_1\in L^3(Q_T)$, the expressions of $\partial_yE(y_1)h_1$ and $\partial_y\bar{F}(y_1)h_1$ correspond to the expressions of $E(y_1)$ and $\bar{F}(y_1)$; only the function $G(y_1)$ is replaced by $\partial_yG(y_1)h_1$. Therefore, we can apply the argumentation in Sec.~\ref{sub:ndop} to $\partial_yE(y_1)$ and $\partial_y\bar{F}(y_1)$ provided that $\partial_yG(y_1)$ fulfills a Lipschitz condition analogous to the one established for $G$. However, being defined as a multiplication with an essentially bounded factor, $\partial_yG(y_1(t))$ is obviously Lipschitz continuous as a function from $L^2(\Om)^2$ to $L^2(\Om)^2$.

Theorem~\ref{satz:ex_allgemein}, applied to $F_1$ and $F_2=0$, finally provides the unique weak solvability of the linearized equation. 
%
%results 
%
%
\section{Conclusions}
In this paper, we analyzed the $PO_4$-$DOP$-model of Parekh et al. \cite{pa05} describing the marine phosphorus cycle. The analysis covered solutions of the original model equations as well as their derivative (or linearization). By investigating the derivative, we prepared a further model analysis, especially concerning optimal parameters. While marine ecosystem models are usually being solved only in a discretized form, we considered the original continuous partial differential equations. Their properties allow to draw conclusions about the validity of the numerical model and the explanatory power of its output. As a result, we found out that both the model equations themselves and their derivative each have a unique weak solution. 

The $PO_4$-$DOP$-model stands as an example for an important class of marine ecosystem models since it contains typical components like saturation function (cf. Sec.~\ref{sub:saturation}) or non-local reaction terms. The latter, for example, appear when sinking processes are modeled via integrals over water columns.% Therefore, results for the $PO_4$-$DOP$-model actually give information about a large class of models.

The results for the two-dimensional model were derived in a much more general context. Thereby, we covered both the derivative and the model equations at the same time and, in addition, other models coupled by reaction terms with both Lipschitz continuous and monotone parts. Combinations of such kinds of reaction terms can appear when different coexisting phenomena have to be described.

We found two existence and uniqueness theorems for marine ecosystem model equations. In particular, we observed that the assumed Lip\-schitz condition holds for non-local reaction terms and reaction terms containing saturation functions.

A special interest lies in Theorem~\ref{satz:ex_speziell} since its proof is based on Banach's Fixed Point Theorem and is hence constructive. Accordingly, the unique solution can be identified with the limit  
of a sequence consisting of solutions of purely monotone problems. In case the monotone parts are zero (as e.g. in the $PO_4$-$DOP$-model), the approximating problems are even linear. The solvability of linear equations is well investigated \cite{la68,evans,tr02}. Especially, if a successful algorithm to solve linear problems is already available a numerical method to compute the unique nonlinear solution can be implemented easily.

The condition of Thm.~\ref{satz:ex_speziell} that $C([0,T];L^2(\Om))^s$ is embedded in the space $Y$ is often fulfilled for the model equations themselves. However, due to the theorem concerning superposition operators cited in Section~\ref{sub:derivatives}, linearized equations are typically defined on $L^p$-spaces into which $C([0,T];L^2(\Om))^s$ is not embedded.
Therefore, Thm.~\ref{satz:ex_allgemein} significantly contributes to the results obtained in this paper. The proof uses Galerkin approximation and treats again the case of combined Lipschitz continuous and monotone reaction terms.

While Thm.~\ref{satz:ex_allgemein} allows a broader variety of spaces $Y$ than Thm.~\ref{satz:ex_speziell} the latter is superior in case that the properties of $F_1$ are not sufficient.
Consider, for instance, the space $Y:=C([0,T];L^2(\Om))^s$ into which $W(0,T)^s$ is not compactly embedded, $F_2= 0$ and $F_1$ not to be weakly continuous. The latter is often true if $F_1$ is nonlinear. Then, Thm.~\ref{satz:ex_allgemein} does not hold because of the missing weak continuity of $F_1$. In the proof of Thm.~\ref{satz:ex_speziell}, however, this deficit is rendered harmless by inserting a fixed element of $z\in Y$ into $F_1$.
 
As a benefit for readers with an applicational background, Lemma~\ref{lem:db_erz_funkt} establishes the connection between the generalized formulation in the theorems and the actual reaction terms $d$ and $b$. 
We carried out all proofs in detail to enable readers to understand the argumentation and adapt it to their own situation if necessary. The proofs may also indicate why a favored reaction term is not allowed and how it could be altered.
%Finally, several examples illustrate possible applications. In particular, the considerations of Sec.~\ref{sec:examples} imply that saturation functions, very common elements of ecosystem models, as well as certain derivatives lead to solvable models.

The results about unique solvability are an important part in the validation of ecosystem models.
As we pointed out in the introduction, a further aspect involves the choice of adequate parameters. 
For the same reasons that inspired us to write this paper, the parameter identification problem should also be an object of mathematical investigation. The first step has already been achieved by solving the linearized equation. Further, questions about existence and uniqueness of optimal parameters will have to be answered. It will also be an interesting  task to find out if one parameter vector necessarily leads to one well-defined model output and thus allows to reconstruct the observational data.

A promising means to answer these questions could be provided by optimal control theory (see Tr\"oltzsch \cite{tr02}). Further research is needed to discover if all aspects of this theory can be transferred to parameter identification problems and which assumptions have to be fulfilled. These results would be the next step towards an improvement of biogeochemical models and thus towards the better understanding of marine ecosystems.  
 
\section*{Acknowledgements}
The research of Christina Roschat was supported by the DFG Cluster Future Ocean.

%\section*{References}
% BibTeX users please use one of
\bibliographystyle{elsarticle-num}      % basic style, author-year citations
\bibliography{meinebibliothek}   % name your BibTeX data base

\end{document}